\documentclass[10pt,a4paper]{scrartcl}%

\usepackage[T1]{fontenc}
\usepackage{graphicx}
\usepackage{amssymb}
\usepackage{amsmath}
\usepackage[utf8]{inputenc}
\usepackage[hyphens]{url}
\usepackage{makeidx}
\usepackage{amsfonts}
\usepackage{ntheorem}

\usepackage{hyperref}
\usepackage[authoryear]{natbib}
\usepackage[french]{babel}

\usepackage{amsbsy}

\theoremstyle{plain}

\newtheorem{proposition}{Proposition}
\theorembodyfont{\upshape}
\newtheorem{definition}{D\'efinition}
\newtheorem{exercice}{Exercice}

\newtheorem{remarque}{Remarque}

\theorembodyfont{\upshape}
\newtheorem{parenthese}[remarque]{Parenth\`ese}
\theoremstyle{nonumberplain}
\theorembodyfont{\upshape}
\newtheorem{demo}{D\'emonstration}
\RequirePackage{ifthen}
\RequirePackage{boites}
\RequirePackage{changepage}
\newcommand{\cadreavant}{\vspace{0.25cm} \begin{adjustwidth}{-0.01\textwidth}{-0.01\textwidth}\nointerlineskip\leavevmode\begin{breakbox}\vspace{-0.25cm}}
\newcommand{\cadreapres}{\vspace{-0.25cm}\end{breakbox}\end{adjustwidth} \vspace{0.25cm}}

\newcommand{\openbox}{\leavevmode
    \hbox to.77778em{%
    \hfil\vrule
    \vbox to.675em{\hrule width.6em\vfil\hrule}%
    \vrule\hfil}}
\renewcommand{\qedsymbol}{\openbox}
\newcommand{\qed}{\hfill \qedsymbol}
\newenvironment{rqpar}[1][]{\begin{quotation}\cadreavant\begin{parenthese}[#1]\small }{\end{parenthese}\cadreapres\end{quotation}}
\newenvironment{proof}[1][]{\begin{quotation}\cadreavant\ifthenelse{\equal{#1}{}}{\begin{demo}}{\begin{demo}[#1]} \small}{\hfill \qed \end{demo}\cadreapres\end{quotation}}

\RequirePackage{dsfont}% pour le 1 avec double barre
\DeclareMathOperator{\argmintmp}{argmin}
\DeclareMathOperator{\argmaxtmp}{argmax}
\newcommand{\argmin}{\mathop{\argmintmp}}
\newcommand{\argmax}{\mathop{\argmaxtmp}}

\newcommand{\bX}{\mathbf{X}}

\newcommand{\cF}{\mathcal{F}}
\newcommand{\cR}{\mathcal{R}}
\newcommand{\cX}{\mathcal{X}}
\newcommand{\cY}{\mathcal{Y}}

\newcommand{\fh}{\widehat{f}}

\newcommand{\fst}{f^{\star}}

\newcommand{\cRh}{\widehat{\mathcal{R}}}

\newcommand{\E}{\mathbb{E}}%% Esperance
\newcommand{\R}{\mathbb{R}} %corps des reels
\newcommand{\N}{\mathbb{N}} %entiers naturels

\renewcommand{\leq}{\leqslant}
\renewcommand{\geq}{\geqslant}

\DeclareMathOperator{\signe}{signe}%signe
\DeclareMathOperator{\crit}{crit}%critere empirique
\DeclareMathOperator{\penal}{pen}%penalite
\DeclareMathOperator{\card}{Card} % cardinal
\DeclareMathOperator{\var}{var} %variance

\newcommand{\cA}{\mathcal{A}}

\newcommand{\cC}{\mathcal{C}}
\newcommand{\cE}{\mathcal{E}}

\newcommand{\cM}{\mathcal{M}}
\newcommand{\cN}{\mathcal{N}}

\newcommand{\cW}{\mathcal{W}}

\newcommand{\mh}{\widehat{m}}

\newcommand{\parenj}[1]{\mathopen{}\left( #1  \right) \mathclose{}} 

\newcommand{\crochj}[1]{\mathopen{}\left[ #1 \right] \mathclose{}} 

\newcommand{\tr}{\mathrm{trace}}

\newcommand{\SAindex}[1]{\index{#1}}

\makeindex

\begin{document}

\title{Validation crois\'ee}
\author{Sylvain Arlot}
\date{8 Mars 2017}
\maketitle

\tableofcontents

\SAindex{hold-out|see{validation simple}}%
\SAindex{validation croisee@validation crois\'ee!blocs@(par) blocs|see{$V$-fold}}% 
\SAindex{validation croisee@validation crois\'ee!repetee@r\'ep\'et\'ee|see{Monte-Carlo}}% 
\SAindex{estimation de densit\'e!par moindres carr\'es|see{moindres carr\'es}}% 
\SAindex{consistance!en selection@en s\'election|see{s\'election de mod\`eles}}% 
\SAindex{histogramme|see{partition}}% 
\SAindex{Nadaraya-Watson (estimateur de)|see{noyau (r\`egle d'apprentissage)}}% 
\SAindex{partition|seealso{arbre de d\'ecision}}%
\SAindex{penalisation@p\'enalisation|see{s\'election de mod\`eles}}%
\SAindex{plus proches voisins|see{$k$ plus proches voisins}}% 
\SAindex{regression@r\'egression!partition@(par) partition|see{partition}}% 
\SAindex{regression@r\'egression!moindres carr\'es@(par) moindres carr\'es|see{moindres carr\'es}}% 
\SAindex{regressogramme@r\'egressogramme|see{partition, r\`egle de r\'egression}}% 
\SAindex{validation croisee@validation crois\'ee|(}%

\newpage

Ce texte aborde deux probl\`emes fondamentaux qui se posent en apprentissage supervis\'e. 

D'une part, une fois que l'on a entra\^in\'e une r\`egle d'apprentissage $\fh$ 
sur un jeu de donn\'ees $D_n$, 
on dispose d'un pr\'edicteur $\fh(D_n)$ mais d'aucune information sur sa qualit\'e. 
Il est n\'ecessairement imparfait, mais \`a quel point ? 
Avant de s'appuyer sur les pr\'evisions qu'il fournit, 
il est indispensable d'\'evaluer son risque.  
\SAindex{risque!estimation@estimation d'un risque}% 

\SAindex{selection estimateurs@s\'election d'estimateurs|(}% 
D'autre part, il est extr\^emement rare que l'on ne puisse utiliser qu'une 
r\`egle d'apprentissage pour un probl\`eme donn\'e. 
Chaque m\'ethode d\'epend en g\'en\'eral d'un ou plusieurs param\`etres 
(un mod\`ele sous-jacent, un niveau de r\'egularisation, 
un nombre de plus proches voisins, etc.) 
dont le choix impacte fortement le risque du pr\'edicteur final. 
Et bien souvent, on dispose de nombreuses m\'ethodes pour un probl\`eme donn\'e 
(par exemple, en classification on peut h\'esiter entre des partitions cubiques, 
les $k$~plus proches voisins, les SVM et les for\^ets al\'eatoires). 
Comment choisir la meilleur m\'ethode et ses param\`etres ? 

\medbreak

La validation crois\'ee fournit une r\'eponse \`a ces deux questions, 
dans un cadre tr\`es g\'en\'eral\footnote{% 
La validation crois\'ee ne s'applique malheureusement pas syst\'ematiquement, 
voir la parenth\`ese~\ref{VC.rk-cadre-le-plus-general} en section~\ref{VC.sec.def.gal}. }. 
Ce texte vise \`a d\'efinir cette famille de proc\'edures et \`a donner les principaux 
\'el\'ements de compr\'ehension disponibles \`a ce jour \`a son sujet. 
Nous donnons volontairement peu de r\'ef\'erences bibliographiques ; 
une bibliographie compl\`ete se trouve dans 
l'article de survol de \citet{arlo_2010} 
et l'article r\'ecent de \citet{arlo_2016}. 

\section{S\'election d'estimateurs} 
\label{VC.sec.pb}
Ce texte se place dans le cadre g\'en\'eral de la pr\'evision, 
tel que pr\'esent\'e par \citet{Arl_2016_JESchap2}, 
dont on utilise les notations et auquel on fait r\'eguli\`erement r\'ef\'erence 
par la suite. 

On peut alors formaliser le probl\`eme de s\'election d'estimateurs\footnote{%%
Compte-tenu de la terminologie introduite par \citet{Arl_2016_JESchap2}, 
il serait plus logique de parler de s\'election de r\`egles d'apprentissage. 
Nous utilisons n\'eanmoins ici l'expression \og{}s\'election d'estimateurs\fg{}, 
plus courante et plus concise.} % fin footnote 
comme suit. 

\medbreak

On dispose d'une collection de r\`egles d'apprentissage 
$(\fh_m)_{m \in \cM_n}$ 
et d'un \'echantillon $D_n$. 
On souhaite pouvoir choisir l'une de ces r\`egles $\fh_{\mh}$ 
\`a l'aide des donn\'ees uniquement. 
Cette question g\'en\'erale recouvre de nombreux exemples : 
\begin{itemize}
\item s\'election de mod\`eles : 
\SAindex{selection modeles@s\'election de mod\`eles}% 
pour tout $m \in \cM_n$, $\fh_m$ est une r\`egle par minimisation du risque empirique 
sur un mod\`ele $S_m$. 
\item choix d'un hyperparam\`etre : 
$m$~d\'esigne alors un ou plusieurs param\`etres r\'eels dont d\'epend la r\`egle $\fh_m$ 
(par exemple, le nombre de voisins $k$ pour les $k$ plus proches voisins, 
ou bien le param\`etre de r\'egularisation $\lambda$ pour les SVM). 
\item choix entre des m\'ethodes de natures diverses 
(par exemple, entre les $k$ plus proches voisins, les SVM et les for\^ets al\'eatoires). 
\end{itemize}

\medbreak

Les enjeux du probl\`eme et approches pour le r\'esoudre sont 
essentiellement les m\^emes que pour le probl\`eme de s\'election de mod\`eles, 
que \citet[section~3.9]{Arl_2016_JESchap2} d\'ecrit en d\'etail. 
Nous n'en rappelons donc ici que les grandes lignes. 

Tout d'abord, il faut pr\'eciser l'objectif (pr\'evision ou identification 
de la \og{}meilleure\fg{} r\`egle $\fh_m$). 
Ce texte se focalise sur l'objectif de pr\'evision : 
on veut minimiser le risque de l'estimateur final $\fh_{\mh(D_n)} (D_n)$ 
--- entra\^in\'e sur l'ensemble des donn\'ees, celles-l\`a m\^emes qui ont servi \`a choisir $\mh(D_n)$. 

Atteindre un tel objectif n\'ecessite d'\'eviter deux d\'efauts principaux : 
le surapprentissage 
\SAindex{surapprentissage}% 
(lorsqu'un pr\'edicteur \og{}colle\fg{} excessivement aux observations, 
ce qui l'emp\^eche de g\'en\'eraliser correctement) 
et le sous-apprentissage 
\SAindex{sous-apprentissage}% 
(quand un pr\'edicteur \og{}lisse\fg{} trop les observations 
et devient incapable de reproduire les variations du pr\'edicteur de Bayes). 
Il s'agit donc de trouver le meilleur compromis entre ces deux extr\^emes. 
Dans de nombreux cas\footnote{% 
Par exemple, pour des estimateurs lin\'eraires en r\'egression, 
des r\`egles par minimisation du risque empirique 
ou des r\`egles par moyennes locales. }, % fin foonote 
ceci se formalise sous la forme d'un compromis biais-variance 
\SAindex{compromis biais-variance}% 
ou approximation-estimation. 

Comment faire ? 
Comme pour la s\'election de mod\`eles, 
on consid\`ere habituellement des proc\'edures de la forme : 
\begin{equation} 
\label{VC.eq.sel-estim.def-mh-gal}
\mh \in \argmin_{m \in \cM_n} \bigl\{ \crit(m;D_n) \bigr\}
\, . 
\end{equation}
On peut les analyser avec le lemme fondamental de l'apprentissage 
\citep[lemme~2]{Arl_2016_JESchap2}. 
Deux strat\'egies principales sont possibles : 
choisir un crit\`ere $\crit(m;D_n)$ 
proche du risque $\cR_P( \fh_m(D_n) )$ simultan\'ement pour tous les 
$m \in \cM_n$, 
ou bien choisir un crit\`ere qui majore le risque. 
La validation crois\'ee suit la premi\`ere strat\'egie. 
Alors, au vu du raisonnement expos\'e par 
\citet[section~3.9]{Arl_2016_JESchap2}, 
il suffit\footnote{Il n'y a cependant pas \'equivalence, 
voir la section~\ref{VC.sec.sel-estim}.} de d\'emontrer que 
$\crit(m;D_n)$ est un bon estimateur 
\SAindex{risque!estimation@estimation d'un risque}% 
du risque\footnote{% 
En toute rigueur, il faut parler d'estimation du risque 
\emph{moyen}, le risque \'etant une quantit\'e al\'eatoire.
Par abus de langage, on parle dans ce texte d'estimation 
(et d'estimateurs) du risque de $\fh_m$, 
et du biais et de la variance de ces estimateurs. 
\`A chaque fois, il est sous-entendu que c'est le risque moyen qu'on estime, 
m\^eme si c'est le risque que l'on souhaite \'evaluer aussi pr\'ecis\'ement que possible. 
} % fin footnote
de $\fh_m(D_n)$ pour en d\'eduire que la proc\'edure d\'efinie par 
\eqref{VC.eq.sel-estim.def-mh-gal} fonctionne bien. 
\SAindex{selection estimateurs@s\'election d'estimateurs|)}% 

C'est pourquoi, apr\`es avoir d\'efini les proc\'edures de validation crois\'ee 
(section~\ref{VC.sec.def}), 
nous commen\c{c}ons par \'etudier leurs propri\'et\'es pour l'estimation du risque 
d'une r\`egle d'apprentissage fix\'ee (section~\ref{VC.sec.estim-risque}), 
avant d'aborder la s\'election d'estimateurs 
(section~\ref{VC.sec.sel-estim}). 
En guise de conclusion, la  section~\ref{VC.sec.concl} consid\`ere 
plusieurs questions pratiques importantes, 
dont celle du choix de la meilleure proc\'edure 
de validation crois\'ee.

\section{D\'efinition} 
\label{VC.sec.def}

\'Etant donn\'e une r\`egle d'apprentissage $\fh_m$, 
un \'echantillon $D_n$ et une fonction de co\^ut $c$, 
la validation crois\'ee estime le risque $\cR_P(\fh_m(D_n))$ 
en se fondant sur le principe suivant : 
on d\'ecoupe l'\'echantillon $D_n$ en deux sous-\'echantillons 
$D_n^{(e)}$ (l'\'echantillon d'entra\^inement) 
et $D_n^{(v)}$ (l'\'echantillon de validation), 
on utilise $D_n^{(e)}$ pour entra\^iner un pr\'edicteur 
$\fh_m(D_n^{(e)})$, 
puis on mesure l'erreur commise par ce pr\'edicteur 
sur les donn\'ees restantes $D_n^{(v)}$. 
Alors, du fait de l'ind\'ependance entre $D_n^{(e)}$ et $D_n^{(v)}$, 
on obtient une bonne \'evaluation\footnote{% 
Comme expliqu\'e en section~\ref{VC.sec.estim-risque}, 
c'est en r\'ealit\'e le risque de $\fh_m(D_n^{(e)})$ que l'on \'evalue, 
d'o\`u un l\'eger biais (beaucoup moins probl\'ematique que celui du risque 
empirique). } % fin footnote
du risque de $\fh_m(D_n)$. 
En particulier, on \'evite l'optimisme excessif\footnote{% 
Les raisons de cet optimisme sont d\'etaill\'ees par 
\citet[section~3.9]{Arl_2016_JESchap2}. } du risque 
empirique $\cRh_n(\fh_m(D_n))$. 
Et l'on peut proc\'eder \`a un ou plusieurs d\'ecoupages 
du m\^eme \'echantillon, 
d'o\`u un grand nombre de proc\'edures de validation crois\'ee possibles.

\subsection{Cas g\'en\'eral}
\label{VC.sec.def.gal}
Dans tout ce texte, 
$D_n = (X_i,Y_i)_{1 \leq i \leq n}$ 
d\'esigne un \'echantillon de variables al\'eatoires ind\'ependantes et de m\^eme loi~$P$. 
On suppose qu'une fonction de co\^ut $c: \cY \times \cY \to \R^+$ 
est fix\'ee et sert \`a d\'efinir le risque $\cR_P$ et le risque empirique 
$\cRh_n$ sur $D_n$ (une quantit\'e d\'efinie par \citet[section~3.1]{Arl_2016_JESchap2}). 

Un sous-ensemble propre\footnote{% 
Un sous-ensemble propre de $\{1, \ldots, n\}$ est une partie 
non-vide de $\{1, \ldots, n\}$ dont le compl\'ementaire est non-vide. 
La terminologie \og{}d\'ecoupage de l'\'echantillon\fg{} n'est pas classique ; 
nous l'utilisons ici pour clarifier l'exposition. 
} % fin footnote
$E$ de $\{1, \ldots, n\}$ est appel\'e \og{}d\'ecoupage\fg{} de l'\'echantillon. 
Il correspond \`a la partition de $D_n$ en deux sous-\'echantillons : 
\[
D_n^{E} := (X_i, Y_i)_{i \in E} 
\qquad \text{et} \qquad 
D_n^{E^c} := (X_i, Y_i)_{i \in \{1 , \ldots, n\} \backslash E} 
\, . 
\]
Pour tout d\'ecoupage $E$ de l'\'echantillon, 
on d\'efinit le risque empirique sur le sous-\'echantillon $D_n^{E}$ par : 
\[
\cRh_n^E: 
f\in \cF \mapsto \frac{1}{\card(E)} \sum_{i \in E}
c \bigl( f(X_i) , Y_i \bigr) 
\, . 
\]
On peut maintenant d\'efinir formellement les estimateurs 
par validation crois\'ee du risque. 

\begin{definition}[Validation crois\'ee]
\label{VC.def.VC}
\SAindex{validation simple}% 
Soit $\fh_m$ une r\`egle d'apprentissage. 
L'estimateur par validation \textup{(}simple\textup{)}\footnote{% 
Le terme anglais pour la validation, ou validation simple, est \og{}hold-out\fg{} : 
il s'agit de l'erreur sur des donn\'ees \og{}mises de c\^ot\'e\fg{} 
au moment de l'entra\^inement. } 
du risque de $\fh_m$ pour l'\'echantillon $D_n$ et le d\'ecoupage $E$ 
est d\'efini par :
\begin{align*} 
\cRh^{\mathrm{val}} (\fh_m;D_n;E) 
&= 
\cRh_n^{E^c} \bigl( \fh_m(D_n^E) \bigr) 
\\
&= 
\frac{1}{\card(E^c)} \sum_{i \in E^c} c \bigl( \fh_m( D_n^E; X_i) , Y_i \bigr) 
\, . 
\end{align*}
On appelle $D_n^E$ 
\SAindex{echantillon apprentissage@\'echantillon d'apprentissage}% 
\SAindex{echantillon entrainement@\'echantillon d'entra\^inement}% 
l'\'echantillon d'entra\^inement\footnote{Le terme \og{}\'echantillon d'apprentissage\fg{} est parfois 
utilis\'e pour d\'esigner l'\'echantillon d'entra\^inement ; 
il arrive aussi qu'on l'utilise pour d\'esigner la r\'eunion de l'\'echantillon 
d'entra\^inement et de l'\'echantillon de validation, lorsqu'une partie des donn\'ees 
est mise de c\^ot\'e dans un \'echantillon test.},  
tandis que $D_n^{E^c}$ est appel\'e 
\SAindex{echantillon validation@\'echantillon de validation}% 
\'echantillon de validation. 

L'estimateur par validation crois\'ee\footnote{En anglais : \og{}cross-validation\fg{}.} 
du risque de $\fh_m$ pour l'\'echantillon $D_n$ et la suite de d\'ecoupages 
$(E_j)_{1 \leq j \leq V}$ est d\'efini par : 
\begin{align*} 
\cRh^{\mathrm{vc}} \bigl( \fh_m; D_n; (E_j)_{1 \leq j \leq V} \bigr) 
= \frac{1}{V} \sum_{j=1}^V \cRh^{\mathrm{val}} ( \fh_m; D_n; E_j ) 
\, . 
\end{align*}
\'Etant donn\'e une famille de r\`egles d'apprentissage 
$(\fh_m)_{m \in \cM_n}$, 
la proc\'edure de s\'election d'estimateurs 
par validation crois\'ee associ\'ee 
est d\'efinie par :
\begin{equation*}
\mh^{\mathrm{vc}} \bigl( D_n; (E_j)_{1 \leq j \leq V} \bigr)  
\in \argmin_{m \in \cM_n} \Bigl\{  
\cRh^{\mathrm{vc}} \bigl( \fh_m; D_n; (E_j)_{1 \leq j \leq V} \bigr)  
\Bigr\}
\, . 
\end{equation*}
\end{definition}

\SAindex{risque!estimation@estimation d'un risque|(}% 
Une erreur courante mais grave est d'utiliser 
\[ 
\cRh^{\mathrm{vc}} \bigl( \fh_m; D_n; (E_j)_{1 \leq j \leq V} \bigr) 
\]
pour estimer le risque de $\fh_m(D_n)$ lorsque la r\`egle d'apprentissage $\fh_m$ d\'epend 
d\'ej\`a elle-m\^eme des donn\'ees. 
Par exemple, si $\fh_m$ est construite sur un sous-ensemble $m$ des covariables disponibles, 
et si ce sous-ensemble a \'et\'e choisi \`a l'aide d'une partie des donn\'ees $D_n$ 
(par une proc\'edure automatis\'ee ou simplement \og{}\`a l'\oe il\fg{}), 
alors on obtient une estimation \emph{fortement biais\'ee} du risque ! 
En g\'en\'eral, cette estimation est tr\`es optimiste,  
conduisant \`a sous-estimer le risque de pr\'evision r\'eel. 

Pour \'eviter ce biais, il faut prendre en compte la \emph{totalit\'e du processus} 
menant des donn\'ees $D_n$ au pr\'edicteur $\fh_m(D_n)$ 
(c'est-\`a-dire, \emph{tout} ce qu'on a fait \`a partir du moment o\`u l'on a eu acc\`es 
\`a au moins une observation). 
Formellement, il faut d\'ecrire comment $m$ d\'epend des donn\'ees, et le noter $\mh(D_n)$. 
On d\'efinit alors $\widetilde{f}: D_n \mapsto \fh_{\mh(D_n)} (D_n)$ 
puis on applique la validation crois\'ee \`a $\widetilde{f}$ 
en calculant : 
\[ \cRh^{\mathrm{vc}} \bigl( \widetilde{f} ; D_n; (E_j)_{1 \leq j \leq V} \bigr)  \, . \]

Le m\^eme probl\`eme se pose quand on veut estimer le risque de l'estimateur 
s\'electionn\'e par la validation crois\'ee (ou par toute autre proc\'edure de s\'election 
d'estimateurs). 
\SAindex{selection estimateurs@s\'election d'estimateurs}% 
Si l'on utilise la valeur (calcul\'ee en cours de proc\'edure) 
\[ 
\cRh^{\mathrm{vc}} \bigl( \fh_{ \mh^{\mathrm{vc}} ( D_n; (E_j)_{1 \leq j \leq V} )  } ; D_n; (E_j)_{1 \leq j \leq V} \bigr) 
= \min_{m \in \cM_n} \cRh^{\mathrm{vc}} \bigl( \fh_m; D_n; (E_j)_{1 \leq j \leq V} \bigr) 
\, ,
\]
alors on commet pr\'ecis\'ement l'erreur mentionn\'ee ci-dessus, 
et l'on sous-estime fortement le risque. 
Il faut donc d\'efinir 
\[
\widetilde{f}^{\mathrm{vc}}: D_n \mapsto \fh_{ \mh^{\mathrm{vc}} ( D_n; (E_{j,n})_{1 \leq j \leq V_n} )  }
\]
(en sp\'ecifiant bien comment la suite de d\'ecoupages 
$(E_{j,n})_{1 \leq j \leq V_n}$ est choisie pour chaque entier $n \geq 1$) 
et lui appliquer la validation crois\'ee en calculant 
\[ \cRh^{\mathrm{vc}} \bigl( \widetilde{f}^{\mathrm{vc}} ; D_n; (E_j)_{1 \leq j \leq V} \bigr)  \, . \]
Dans le cas de la validation simple, 
ceci conduit \`a un d\'ecoupage de l'\'echantillon en \emph{trois sous-\'echantillons} : 
un 
\SAindex{echantillon entrainement@\'echantillon d'entra\^inement}% 
\'echantillon d'entra\^inement $D_n^{E}$, 
un \'echantillon de validation $D_n^{V}$ 
\SAindex{echantillon validation@\'echantillon de validation}% 
--- pour choisir parmi les $\fh_m(D_n^E)$ ---, 
et un 
\SAindex{echantillon test@\'echantillon test}% 
\'echantillon test $D_n^{T}$ 
--- pour \'evaluer le risque de l'estimateur final $\fh_{\mh(D_n^E, D_n^V)} (D_n^E)$ ---, 
o\`u $E$, $V$ et $T$ forment une partition de $\{1, \ldots, n\}$. 

Signalons toutefois que d'autres approches permettent d'\'eviter la n\'ecessit\'e 
de recourir \`a un d\'ecoupage de l'\'echantillon en trois, 
en particulier le \og{}reusable hold-out\fg{} r\'ecemment propos\'e par 
\citet{dwor_2015}, qui repose sur l'id\'ee de n'acc\'eder \`a l'\'echantillon de validation 
que par l'interm\'ediaire d'un m\'ecanisme de 
confidentialit\'e diff\'erentielle\footnote{Le terme anglais est 
\og{}differential privacy\fg{}. }. 
\SAindex{risque!estimation@estimation d'un risque|)}% 

\begin{rqpar}[Cadre plus g\'en\'eral]
\label{VC.rk-cadre-le-plus-general} 
On peut d\'efinir la validation crois\'ee hors du cadre de pr\'evision 
ou pour une fonction de risque plus g\'en\'erale que celle d\'efinie 
par \citet{Arl_2016_JESchap2}. 
Il suffit en effet que le risque d'un \'el\'ement $f$ de l'ensemble $\cF$ des sorties 
possibles d'une r\`egle d'apprentissage v\'erifie : 
\begin{equation}
\label{VC.hyp.cadre-le-plus-general} 
\forall n \geq 1, \qquad \cR_P(f) = \E \bigl[ \cE(f;D_n) \bigr]
\end{equation}
o\`u $D_n$ est un \'echantillon de $n$ variables al\'eatoires ind\'ependantes et de loi~$P$,  
et $\cE$ est une fonction \`a valeurs r\'eelles, prenant en entr\'ee un \'el\'ement de $\cF$ 
et un \'echantillon de taille quelconque. 
La fonction $\cE$ mesure l'\og{}ad\'equation\fg{} entre $f$ et l'\'echantillon~$D_n$. 
\SAindex{validation simple}% 
L'estimateur par validation simple se d\'efinit alors par : 
\[
\cRh^{\mathrm{val}} (\fh_m;D_n;E) 
:= 
\cE  \Bigl( \fh_m \bigl( D_n^{E^c} \bigr) ; D_n^{E^c} \Bigr)
\]
et l'estimateur par validation crois\'ee s'en d\'eduit. 
Dans le cas de la pr\'evision, \eqref{VC.hyp.cadre-le-plus-general} 
est v\'erifi\'ee avec : 
\[
\cE \bigl( f; (X_i,Y_i)_{1 \leq i \leq n} \bigr) 
= \frac{1}{n} \sum_{i=1}^n c\bigl( f(X_i) , Y_i \bigr) = \cRh_n(f)
\, . 
\]
La relation \eqref{VC.hyp.cadre-le-plus-general}  est \'egalement v\'erifi\'ee dans d'autres cadres. 
Par exemple, en estimation de densit\'e, 
\SAindex{estimation de densit\'e}% 
il est classique de consid\'erer le risque quadratique 
\SAindex{risque!quadratique}% 
d\'efini par \eqref{VC.hyp.cadre-le-plus-general} avec :  
\[ 
\cE \bigl( f; (X_i)_{1 \leq i \leq n} \bigr) 
= \lVert f \rVert^2_{L^2} - \frac{2}{n} \sum_{i=1}^n f(X_i)
\, .
\]
L'exc\`es de risque correspondant est $\lVert f - \fst \rVert^2_{L^2}$ 
o\`u $\fst$ est la densit\'e (inconnue) des observations 
\citep{arlo_2016}. 
On peut aussi obtenir la distance de Kullback-Leibler entre $f$ et $ \fst$ comme exc\`es de risque 
en d\'efinissant le risque par 
\eqref{VC.hyp.cadre-le-plus-general} avec : 
\[ 
\cE \bigl( f; (X_i)_{1 \leq i \leq n} \bigr) 
= - \frac{1}{n} \sum_{i=1}^n \ln \bigl( f(X_i) \bigr)
\, ,  
\]
qui est l'oppos\'e de la log-vraisemblance de~$f$ au vu de l'\'echantillon $(X_i)_{1 \leq i \leq n}$. 
\end{rqpar}

\subsection{Exemples}
\label{VC.sec.def.ex}
Comme l'indique la d\'efinition~\ref{VC.def.VC}, 
il y a autant de proc\'edures de validation crois\'ee 
que de suites de d\'ecoupages $(E_j)_{1 \leq j \leq V}$. 
Au sein de cette grande famille, 
certaines proc\'edures sont toutefois plus classiques que d'autres. 

La plupart des proc\'edures utilis\'ees v\'erifient les deux hypoth\`eses suivantes : 
\begin{gather}
\label{VC.eq.hyp-Ind}
\tag{\ensuremath{\mathbf{Ind}}}
(E_j)_{1 \leq j \leq V} 
\text{ est ind\'ependante de } D_n 
\\
\label{VC.eq.hyp-Reg}
\tag{\ensuremath{\mathbf{Reg}}}
\card(E_1) = \card(E_2) = \cdots = \card(E_V) 
= n_e \in \{ 1, \ldots, n-1 \} 
\, . 
\end{gather}
On suppose toujours dans ce texte que 
\eqref{VC.eq.hyp-Ind} et \eqref{VC.eq.hyp-Reg} 
sont v\'erifi\'ees. 

\begin{rqpar}[Sur l'hypoth\`ese~\eqref{VC.eq.hyp-Ind}]
\label{VC.rk.hyp-Ind}
L'hypoth\`ese \eqref{VC.eq.hyp-Ind} garantit que l'\'echantillon 
d'entra\^inement $D_n^{E_j}$ et l'\'echantillon de validation 
$D_n^{E_j^c}$ sont ind\'ependants pour tout $j$, 
ce qui est crucial pour l'analyse men\'ee 
en section~\ref{VC.sec.estim-risque.biais} notamment. 
Cependant, il est parfois sugg\'er\'e de choisir $(E_j)_{1 \leq j \leq V}$ 
en utilisant l'\'echantillon $D_n$ pour diverses raisons : 
\begin{itemize}
\item pour que l'ensemble du support de $P_X$ 
soit repr\'esent\'e dans chaque \'echantillon d'entra\^inement 
et chaque \'echantillon de validation. 
\item en classification, pour que toutes les classes soient 
\SAindex{classification supervis\'ee}% 
repr\'esent\'ees dans chaque \'echantillon d'entra\^inement 
et chaque \'echantillon de validation 
(en particulier lorsque les effectifs des classes sont tr\`es 
d\'es\'equilibr\'es, 
\SAindex{classification supervis\'ee!multiclasse}% 
ou lorsque le nombre de classes est grand). 
\end{itemize}
Nous ne connaissons cependant pas de r\'esultat th\'eorique 
justifiant l'int\'er\^et d'un tel choix. 
En r\'egression, dans les simulations de \citet[section~6.2]{brei_1992}, 
une telle strat\'egie n'a pas d'impact sur les performances. 
\end{rqpar}

\begin{rqpar}[\'Echantillon ordonn\'e et hypoth\`ese~\eqref{VC.eq.hyp-Ind}]
%%\label{VC.rk.ech-ord-et-hyp-Ind}
%
Souvent, en pratique, on dispose d'un \'echantillon 
\og{}ordonn\'e\fg{}. 
Par exemple, lorsque $\cX = \R$, on a souvent 
$X_1 \leq X_2 \leq \cdots \leq X_n$ 
(ce qui signifie, si $D_n$ contient bien les r\'ealisations de $n$ variables 
al\'eatoires ind\'ependantes de loi~$P$, que l'\'echantillon initial a \'et\'e r\'eordonn\'e). 
Alors, si l'on veut utiliser une proc\'edure de validation crois\'ee 
v\'erifiant l'hypoth\`ese~\eqref{VC.eq.hyp-Ind}, 
il faut prendre soin d'appliquer au pr\'ealable une permutation al\'eatoire 
uniforme des indices $\{1, \ldots, n\}$. 
Sinon, le d\'ecoupage $E = \{1, \ldots, n/2\}$ 
ne peut pas \^etre consid\'er\'e ind\'ependant de $D_n$. 
Notons toutefois que ceci n'est pas n\'ecessaire pour les proc\'edures 
\og{}leave-one-out\fg{} et \og{}leave-$p$-out\fg{}.  
\end{rqpar}

\begin{rqpar}[Sur l'hypoth\`ese~\eqref{VC.eq.hyp-Reg}]
%%\label{VC.rk.hyp-Reg}
%
Nous ne connaissons pas d'argument th\'eorique en faveur des proc\'edures 
de validation crois\'ee v\'erifiant \eqref{VC.eq.hyp-Reg}, 
si ce n'est qu'elles sont plus simples \`a analyser que les autres. 
Il n'emp\^eche que l'hypoth\`ese \eqref{VC.eq.hyp-Reg} 
est toujours v\'erifi\'ee (au moins approximativement) 
dans les applications. 
Les r\'esultats th\'eoriques mentionn\'es dans ce texte 
restent valables (au premier ordre) 
lorsque \eqref{VC.eq.hyp-Reg} n'est v\'erifi\'ee qu'approximativement, 
cas in\'evitable si l'on utilise la validation crois\'ee 
\og{}$V$-fold\fg{} avec $n$ non-divisible par~$V$.   
\SAindex{validation croisee@validation crois\'ee!V-fold@$V$-fold}% 
\end{rqpar}

\medbreak

Parmi les proc\'edures v\'erifiant les hypoth\`eses \eqref{VC.eq.hyp-Ind} et 
\eqref{VC.eq.hyp-Reg}, on peut distinguer deux approches. 
Soit l'on consid\`ere la suite de \emph{tous} les d\'ecoupages 
$E_j$ tels que $\card(E_j) = n_e$. 
Lorsque $n_e = n-1$, on obtient le 
\SAindex{validation croisee@validation crois\'ee!leave-one-out}% 
leave-one-out\footnote{% 
En fran\c{c}ais, la proc\'edure leave-one-out peut 
\^etre nomm\'ee validation crois\'ee \og{}tous sauf un\fg{} : chaque d\'ecoupage laisse 
exactement une observation hors de l'\'echantillon 
d'entra\^inement.  
En anglais, on trouve aussi les noms \og{}delete-one cross-validation\fg{}, 
\og{}ordinary cross-validation\fg{}, 
et m\^eme parfois simplement \og{}cross-validation\fg{}. 
Dans le cas o\`u $\fh_m$ est un estimateur des moindres carr\'es en r\'egression lin\'eaire 
\citep[exemple~4 en section~3.2]{Arl_2016_JESchap2}, 
le leave-one-out est parfois appel\'e PRESS (Prediction Sum of Squares), ou PRESS de Allen ; 
\SAindex{PRESS}% 
notons que ce terme d\'esigne parfois directement 
la formule simplifi\'ee \eqref{VC.eq.LOO-reg-lin} que l'on peut 
d\'emontrer dans ce cadre.} : % fin footnote
\begin{align*}
\cRh^{\mathrm{loo}}( \fh_m ; D_n ) 
:=& 
\cRh^{\mathrm{vc}} \Bigl( \fh_m ; D_n; \bigl( \{ j \}^c \bigr)_{1 \leq j \leq n} \Bigr)
\\
=& \frac{1}{n} \sum_{j=1}^n c \bigl( \fh_m( D_n^{(-j)} ; X_j ) , Y_j \bigr)
\end{align*}
o\`u $D_n^{(-j)} = (X_i,Y_i)_{1 \leq i \leq n, \, i \neq j}$. 
Dans le cas g\'en\'eral, en posant $p = n-n_e$, 
on obtient le 
\SAindex{validation croisee@validation crois\'ee!leave-p-out@leave-$p$-out}% 
leave-$p$-out\footnote{% 
En anglais, on trouve aussi les termes 
\og{}delete-$p$ cross-validation\fg{} 
et \og{}delete-$p$ multifold cross-validation\fg{}.} : % fin footnote
\begin{align*}
\cRh^{\mathrm{lpo}}( \fh_m ; D_n ; p) 
:=& 
\cRh^{\mathrm{vc}} \Bigl( \fh_m ; D_n; \bigl\{ E \subset \{1, \ldots, n\}
 \, / \, \card(E) = n-p \bigr\} \Bigr)
\, . 
\end{align*}

\medbreak

En pratique, il est souvent trop co\^uteux algorithmiquement 
(voire impossible) 
d'utiliser le leave-one-out ou le leave-$p$-out. 
Une deuxi\`eme approche est donc n\'ecessaire : 
n'explorer que partiellement l'ensemble des 
$\binom{n}{n_e}$ d\'ecoupages possibles 
avec un \'echantillon d'entra\^inement de taille~$n_e$. 

Consid\'erer un seul d\'ecoupage am\`ene \`a la validation simple 
\SAindex{validation simple}% 
ou \og{}hold-out\fg{} ; 
toutes ces proc\'edures sont \'equivalentes car on a fait l'hypoth\`ese~\eqref{VC.eq.hyp-Ind}. 
En revanche, d\`es que l'on consid\`ere un nombre de d\'ecoupages 
$V \in \bigl[ 2 , \binom{n}{n_e} \bigr[$, 
%%$V$ tel que \[ 2 \leq V < \binom{n}{n_e} \, , \]
plusieurs proc\'edures non-\'equivalentes sont possibles. 

\SAindex{validation croisee@validation crois\'ee!V-fold@$V$-fold}% 
La plus classique est la validation crois\'ee par blocs, 
appel\'ee validation crois\'ee \og{}$V$-fold\fg{} 
ou \og{}$k$-fold\fg{}. 
On se donne une partition $(B_j)_{1 \leq j \leq V}$ de 
$\{1, \ldots, n\}$ en $V$ blocs de m\^eme taille\footnote{% 
Il faut prendre les $B_j$ de m\^eme taille pour avoir 
l'hypoth\`ese \eqref{VC.eq.hyp-Reg}. 
Lorsque $V$ ne divise pas $n$, il suffit de les prendre 
de tailles \'egales \`a un \'el\'ement pr\`es ; les performances th\'eoriques et 
pratiques sont alors similaires \`a ce que l'on a quand  
\eqref{VC.eq.hyp-Reg} est v\'erifi\'ee exactement.}, % fin footnote
puis on proc\`ede \`a un \og{}leave-one-out par blocs\fg{}, 
c'est-\`a-dire, on utilise la suite de d\'ecoupages 
$(B_j^c)_{1 \leq j \leq V}$ : 
\begin{align*}
\cRh^{\mathrm{vf}} \bigl( \fh_m ; D_n ; (B_j)_{1 \leq j \leq V} \bigr) 
:=& \cRh^{\mathrm{vc}} \bigl( \fh_m ; D_n ; (B_j^c)_{1 \leq j \leq V} \bigr) 
\\
=& \frac{1}{V} \sum_{j=1}^V \cRh_n^{B_j} \biggl( \fh_m \Bigl( D_n^{B_j^c} \Bigr) \biggr)
\, . 
\end{align*}

\SAindex{validation croisee@validation crois\'ee!Monte-Carlo}%  
On peut \'egalement proc\'eder \`a une validation crois\'ee Monte-Carlo 
(ou validation crois\'ee r\'ep\'et\'ee), 
en choisissant $E_1, \ldots, E_V$ al\'eatoires, 
ind\'ependants et de loi uniforme sur l'ensemble 
des parties de taille $n_e$ de $\{1, \ldots, n\}$. 

\begin{remarque}[$V$-fold ou Monte-Carlo ?]
\label{VC.sec.def.ex.rk-VF-ou-MC}
\SAindex{validation croisee@validation crois\'ee!V-fold@$V$-fold}% 
\SAindex{validation croisee@validation crois\'ee!Monte-Carlo}% 
On discute dans la suite 
les m\'erites de ces deux approches. 
Intuitivement, on peut d\'ej\`a dire que la validation crois\'ee $V$-fold 
pr\'esente l'avantage de faire un usage \og{}\'equilibr\'e\fg{} des donn\'ees : 
chaque observation est utilis\'ee exactement $V-1$ fois 
pour l'entra\^inement et une fois pour l'apprentissage. 
Ce n'est en rien garanti avec la validation crois\'ee Monte-Carlo. 
En revanche, on peut s'interroger sur les inconv\'enients 
de toujours utiliser ensemble (soit pour l'entra\^inement, soit pour la validation) 
les observations d'un m\^eme bloc. 
L'approche \og{}Monte-Carlo\fg{}, par son caract\`ere al\'eatoire, 
permet d'\'eviter d'\'eventuels biais induits par ce lien entre observations. 
\SAindex{validation croisee@validation crois\'ee!incompl\`ete \'equilibr\'ee}% 
Notons qu'il existe une mani\`ere d'\'eviter ces deux \'ecueils : 
la validation crois\'ee incompl\`ete \'equilibr\'ee\footnote{% 
En anglais, on parle de \og{}balanced-incomplete cross-validation\fg{}, 
qui s'appuie sur la notion de \og{}balanced-incomplete block-design\fg{}.} % fin footnote 
\citep[section~4.3.2]{arlo_2010}, 
qui pr\'esente l'inconv\'enient de n'\^etre possible que pour d'assez grandes valeurs de $V$.
\end{remarque}

Signalons enfin que bien d'autres proc\'edures de 
validation crois\'ee \og{}non-exhaustives\fg{} existent. 
\SAindex{validation croisee@validation crois\'ee!V-fold repetee@$V$-fold r\'ep\'et\'ee}%  
Par exemple, avec la validation crois\'ee $V$-fold r\'ep\'et\'ee, 
on choisit plusieurs partitions $(B_j^{\ell})_{1 \leq j \leq V}$, 
$\ell \in \{1, \ldots, L\}$, et l'on explore l'ensemble des d\'ecoupages 
\[
\bigl\{  ( B_j^{\ell} )^c  \,/ \, j \in \{ 1 , \ldots, V \} 
    \text{ et } \ell \in \{1, \ldots, L \} \bigr\} 
\, . 
\]

\subsection{Astuces algorithmiques}
\label{VC.sec.def.algo}
La complexit\'e algorithmique du calcul \og{}na\"if\fg{} de 
\[
\cRh^{\mathrm{vc}} \bigl( \fh_m; D_n; (E_j)_{1 \leq j \leq V} \bigr) 
= \frac{1}{V} \sum_{j=1}^V \cRh_n^{E_j^c} \bigl( \fh_m(D_n^{E_j}) \bigr) 
\]
est de l'ordre de $V$ fois celle de l'entra\^inement de $\fh_m$ 
sur un \'echantillon de taille $n_e$ (c'est en g\'en\'eral le plus co\^uteux), 
plus l'\'evaluation de $\fh_m(D_n^{E_j})$ en $n-n_e$ points. 
Il est cependant parfois possible de faire beaucoup plus rapide. 

\medbreak

Tout d'abord, on dispose dans certains cas de formules closes 
pour l'estimateur par validation crois\'ee du risque de $\fh_m$, 
au moins pour le leave-one-out et le leave-$p$-out\footnote{% 
Au vu des r\'esultats des sections~\ref{VC.sec.estim-risque} et 
\ref{VC.sec.sel-estim}, en particulier la 
proposition~\ref{VC.pro.var-ho-geq-lpo} en section~\ref{VC.sec.estim-risque.var} 
qui montre que la variance de la validation crois\'ee est minimale 
pour le leave-$p$-out --- \`a $n_e$ fix\'ee ---, 
il semble inutile de consid\'erer un autre type de validation crois\'ee 
quand on sait calculer rapidement tous les estimateurs 
par leave-$p$-out. } \citep[section~9]{arlo_2010}. 
\SAindex{validation croisee@validation crois\'ee!leave-one-out}% 
\SAindex{validation croisee@validation crois\'ee!leave-p-out@leave-$p$-out}% 
\SAindex{moindres carres@moindres carr\'es!regression lineaire@r\'egression lin\'eaire|(}% 
Par exemple, si $\fh_m$ est un estimateur des moindres carr\'es en r\'egression lin\'eaire, 
la formule de Woodbury 
\citep[section~2.7]{press_etal_NRC1992} 
permet de d\'emontrer \citep{golu_1979} : 
\begin{align}
\notag 
\cRh^{\mathrm{loo}} \bigl( \fh_m ; (x_i,y_i)_{1 \leq i \leq n} \bigr) 
&= \frac{1}{n} \sum_{i=1}^n \frac{\Bigl[ y_i - \fh_m \bigl( (x_i,y_i)_{1 \leq i \leq n} ; x_i \bigr) \Bigr]^2}{1 - \mathbf{H}_{i,i}}
\\
\label{VC.eq.LOO-reg-lin}
&= \frac{1}{n} \sum_{i=1}^n \frac{\bigl( y_i - (\mathbf{H} \mathbf{y})_i \bigr)^2}{1 - \mathbf{H}_{i,i}}
\\
\notag \text{o\`u} \qquad 
\mathbf{H} 
&= \bX \bigl( \bX^{\top} \bX \bigr)^{-1} \bX^{\top} 
\qquad \text{et} \qquad 
\mathbf{y} = (y_i)_{1 \leq i \leq n} \in \R^n
\, . 
\end{align}
Le calcul de l'estimateur leave-one-out avec \eqref{VC.eq.LOO-reg-lin} est 
aussi co\^uteux que celui 
d'\emph{un seul} entra\^inement de $\fh_m$ sur $D_n$, 
via le calcul de la matrice~$\mathbf{H}$. 

\begin{rqpar}[Validation crois\'ee g\'en\'eralis\'ee]
\label{VC.rk.GCV}
\SAindex{validation croisee@validation crois\'ee!generalisee@g\'en\'eralis\'ee (GCV)}% 
La formule close \eqref{VC.eq.LOO-reg-lin}  obtenue pour le leave-one-out 
en r\'egression lin\'eaire a conduit \`a en d\'efinir 
une version \og{}invariante par rotation\fg{} \citep{golu_1979}, 
appel\'ee validation crois\'ee g\'en\'eralis\'ee ou GCV 
(de l'anglais \og{}generalized cross-validation\fg{}).  
Par rapport \`a la formule \eqref{VC.eq.LOO-reg-lin}, 
les d\'enominateurs $1-\mathbf{H}_{i,i}$ 
sont remplac\'es par $1 - n^{-1} \tr(\mathbf{H})$ : 
\[
\mathrm{GCV} ( \fh_m ; D_n ) 
= \frac{\lVert \mathbf{y} - \mathbf{H} \mathbf{y} \rVert^2 }{n - \tr(\mathbf{H}) } 
\, . 
\]
Ce crit\`ere s'applique, plus g\'en\'eralement, 
\`a tout estimateur \og{}lin\'eaire\fg{}  
en r\'egression avec le co\^ut quadratique, 
notamment les $k$ plus proches voisins 
\SAindex{k plus proches voisins@$k$ plus proches voisins}% 
\citep[section~5.4]{Arl_2016_JESchap2}, 
les estimateurs par noyau 
\SAindex{noyau (r\`egle d'apprentissage)}% 
\citep[section~5.5]{Arl_2016_JESchap2}
et les estimateurs ridge.  
\SAindex{regression@r\'egression!ridge}% 
\citet{efro_1986} explique pourquoi, 
malgr\'e son nom, GCV est beaucoup plus proche 
des crit\`eres $C_p$ et $C_L$ de Mallows 
que de la validation crois\'ee proprement dite. 
\end{rqpar}
\SAindex{moindres carres@moindres carr\'es!regression lineaire@r\'egression lin\'eaire|)}% 

\medbreak

Par ailleurs, m\^eme en l'absence de formule close, 
on peut r\'eduire la complexit\'e algorithmique de la validation crois\'ee 
en entra\^inant d'abord $\fh_m$ sur l'\'echantillon $D_n$ tout entier 
(une fois pour toutes), 
puis, pour chaque d\'ecoupage $E_j$, 
en \og{}mettant \`a jour\fg{} $\fh_m(D_n)$ afin d'obtenir 
$\fh_m(D_n^{E_j^c})$. 
Lorsque cette mise \`a jour est efficace, 
le gain algorithmique est important. 
Cette id\'ee s'applique dans plusieurs cadres, 
dont l'analyse discriminante (lin\'eaire ou quadratique) 
\SAindex{analyse discriminante!lineaire@lin\'eaire}% 
\SAindex{analyse discriminante!quadratique}% 
et les $k$~plus proches voisins 
\SAindex{k plus proches voisins@$k$ plus proches voisins}% 
\citep[section~9]{arlo_2010}. 
\SAindex{moindres carres@moindres carr\'es!regression lineaire@r\'egression lin\'eaire}% 
Consid\'erons ici \`a nouveau l'estimateur des moindres carr\'es en r\'egression lin\'eaire. 
Son calcul n\'ecessite d'inverser la matrice $\bX^{\top} \bX$ de taille $p$, 
ce qui a un co\^ut de l'ordre de $p^3$. 
Lorsque $n-n_e$ est significativement plus petit que~$n$, 
on peut utiliser la formule de Woodbury 
\citep[section~2.7]{press_etal_NRC1992} pour d\'eduire 
$(\bX_{(E_j^c)}^{\top} \bX_{(E_j^c)})^{-1}$ de $(\bX^{\top} \bX)^{-1}$ 
\`a moindres frais, 
\[ 
\text{o\`u} \qquad 
\bX_{(E_j^c)} = (x_{i,k})_{1 \leq i \leq n , k \in E_j^c} 
\, . 
\]
\SAindex{regression@r\'egression!ridge}% 
La formule de Woodbury est \'egalement utile en r\'egression ridge, 
o\`u l'essentiel du temps de calcul de l'estimateur est consacr\'e 
\`a l'inversion de la m\^eme matrice $\bX^{\top} \bX$.

\medbreak

\SAindex{selection estimateurs@s\'election d'estimateurs|(}% 
Enfin, dans un contexte de s\'election d'estimateurs, 
il n'est pas toujours n\'ecessaire de calculer  
\[
\cRh^{\mathrm{vc}} \bigl( \fh_m; D_n; (E_j)_{1 \leq j \leq V} \bigr) 
\]
pour chaque $m \in \cM_n$. 
Les valeurs de l'erreur de validation obtenues sur les premiers d\'ecoupages 
$E_1$, $E_2$, $E_3$, $\ldots$ peuvent suffire \`a \'eliminer certains $\fh_m$ 
(et donc de gagner en temps de calcul),  
sans trop perdre sur la qualit\'e du pr\'edicteur final 
\citep{krue_2015}.

\subsection{Variantes}
\label{VC.sec.def.var}
Pour le probl\`eme de s\'election d'estimateurs, 
il y a plusieurs variantes de la validation crois\'ee, 
qui ne suivent pas la d\'efinition~\ref{VC.def.VC} 
mais reposent sur le m\^eme principe d'entra\^inement et validation 
selon plusieurs d\'ecoupages successifs.

\medbreak

\citet{yang_2006,yang_2007} propose la 
\SAindex{validation croisee@validation crois\'ee!vote@(par) vote}% 
\og{}validation 
crois\'ee par vote\fg{}\footnote{% 
Le terme anglais est \og{}cross-validation with voting\fg{}. 
Par opposition, Yang nomme \og{}cross-validation with averaging\fg{} 
la validation crois\'ee habituelle, celle de la d\'efinition~\ref{VC.def.VC}. }, 
lorsque l'objectif est d'\emph{identifier} la meilleure r\`egle d'apprentissage 
(comme expliqu\'e par \citet[section~3.9]{Arl_2016_JESchap2}; 
voir aussi la remarque~\ref{VC.rk.identification} 
en section~\ref{VC.sec.concl.choix}). 
Pour chacun des $j \in \{1, \ldots, V\}$, 
on s\'electionne un estimateur par validation simple : 
\begin{equation}
\label{VC.eq.VC-vote.mhj}
\mh_j \in \argmin_{m \in \cM_n} \bigl\{ \cRh^{\mathrm{val}} (\fh_m;D_n;E_j) \bigr\} 
\, . 
\end{equation}
Ensuite, on r\'ealise un vote majoritaire parmi les $\mh_j$ 
pour d\'eterminer $\mh$. 
Clairement, ceci n'a de sens que lorsque $\cM_n$ est discret. 
Yang propose cette variante dans un contexte o\`u $\cM_n$ est fini 
et de petite taille. 
Supposons par exemple que l'on veut choisir, pour un probl\`eme de classification, 
entre les plus proches voisins, la r\'egression logistique 
et les for\^ets al\'eatoires. 
Si les param\`etres de chaque m\'ethode sont choisis 
par une boucle interne de validation crois\'ee, 
on a $\card(\cM_n) = 3$. 
Cela fait donc sens d'effectuer un vote majoritaire 
parmi les $\mh_j$ obtenus sur $V \gg 3$ d\'ecoupages diff\'erents. 

\medbreak

\SAindex{validation croisee@validation crois\'ee!agregee@agr\'eg\'ee}% 
La validation crois\'ee agr\'eg\'ee\footnote{En anglais, on utilise les termes 
\og{}CV bagging\fg{} ou \og{}averaging cross-validation\fg{}, 
pour un ensemble de m\'ethodes similaires \`a celle qui est d\'ecrite ici. } 
est une variante largement utilis\'ee en pratique 
pour ses bonnes performances en pr\'evision, 
mais peu mentionn\'ee dans la litt\'erature \citep{jung_2015,mail_2016}. 
L'id\'ee est de ne pas \emph{s\'electionner} l'un des $\fh_m$ 
\SAindex{agregation@agr\'egation}% 
mais d'en \emph{combiner} plusieurs pour obtenir un pr\'edicteur 
encore plus performant (parfois meilleur que le choix oracle, 
d'apr\`es des r\'esultats exp\'erimentaux). 
Comme pour la validation crois\'ee par vote, 
pour chaque d\'ecoupage $E_j$, $j \in \{1, \ldots, V\}$, on s\'electionne 
$\mh_j$ d\'efini par \eqref{VC.eq.VC-vote.mhj}. 
%%%%
%%%\[
%%%\mh_j \in \argmin_{m \in \cM_n} \bigl\{ \cRh^{\mathrm{val}} (\fh_m;D_n;E_j) \bigr\} 
%%%\, . 
%%%\]
%
Ensuite, on construit un pr\'edicteur $\widetilde{f} (D_n)$ 
en agr\'egeant les pr\'edicteurs $\fh_{\mh_j} (D_n)$ obtenus 
avec chaque d\'ecoupage. 
En r\'egression, on fait une moyenne : 
\[ 
\widetilde{f} (D_n) = \frac{1}{V} \sum_{j=1}^V \fh_{\mh_j} (D_n)
\, . 
\]
En classification, on proc\`ede \`a un vote majoritaire : 
\[
\widetilde{f} (D_n ; x) 
= \argmax_{y \in \cY} \card \bigl\{ j \in \{1, \ldots, V\} \, / \, 
\fh_{\mh_j} (D_n ; x) = y \bigr\}
\, . 
\]
Cette id\'ee d'agr\'egation est \`a rapprocher du bagging, 
\SAindex{bagging}% 
bien que la validation crois\'ee agr\'eg\'ee ne co\"incide pas exactement 
avec le bagging appliqu\'e \`a la validation simple. 
\SAindex{selection estimateurs@s\'election d'estimateurs|)}% 

\section{Estimation du risque : biais et variance}
\label{VC.sec.estim-risque}
\SAindex{risque!estimation@estimation d'un risque|(}% 

Si l'on utilise la validation crois\'ee pour estimer le risque (moyen) 
d'une r\`egle d'apprentissage $\fh_m$ fix\'ee, 
il est naturel de s'int\'eresser \`a deux quantit\'es : 
son biais et sa variance.

\subsection{Biais}
\label{VC.sec.estim-risque.biais}

Sous les hypoth\`eses \eqref{VC.eq.hyp-Ind} et \eqref{VC.eq.hyp-Reg}, 
l'esp\'erance d'un crit\`ere par validation crois\'ee g\'en\'eral 
se calcule ais\'ement : 
\begin{equation} 
\label{VC.eq.E.gal}
\E \Bigl[ \cRh^{\mathrm{vc}} \bigl( \fh_m ; D_n ; (E_j)_{1 \leq j \leq V} \bigr) \Bigr]
= \E \Bigl[ \cR_P \bigl( \fh_m (D_{n_e}) \bigr) \Bigr]
\end{equation} 
o\`u $D_{n_e}$ d\'esigne un \'echantillon de $n_e$ variables ind\'ependantes 
et de m\^eme loi~$P$. 

\begin{proof}
Par d\'efinition de la validation crois\'ee, on a : 
\[
\E \Bigl[ \cRh^{\mathrm{vc}} \bigl( \fh_m ; D_n ; (E_j)_{1 \leq j \leq V} \bigr) \Bigr]
= \frac{1}{V} \sum_{j=1}^V 
\E \biggl[ \cRh_n^{E_j^c} \Bigl( \fh_m \bigl(D_{n}^{E_j} \bigr) \Bigr) \biggr]
\, . 
\]
D'apr\`es l'hypoth\`ese \eqref{VC.eq.hyp-Ind}, 
$D_n^{E_j}$ et $D_n^{E_j^c}$ sont deux \'echantillons ind\'ependants 
de variables ind\'ependantes de loi~$P$, 
donc 
\[
\E \Bigl[ \cRh^{\mathrm{vc}} \bigl( \fh_m ; D_n ; (E_j)_{1 \leq j \leq V} \bigr) \Bigr]
= \frac{1}{V} \sum_{j=1}^V 
\E \biggl[ \cR_P \Bigl( \fh_m \bigl(D_{n}^{E_j} \bigr) \Bigr) \biggr]
\, . 
\]
Comme l'hypoth\`ese \eqref{VC.eq.hyp-Reg} garantit que les $E_j$ sont tous de m\^eme taille, on a 
\[
\E \biggl[ \cR_P \Bigl( \fh_m \bigl(D_{n}^{E_j} \bigr) \Bigr) \biggr]
= 
\E \Bigl[ \cR_P \bigl( \fh_m (D_{n_e}) \bigr) \Bigr]
\]
pour tout $j$, d'o\`u le r\'esultat. 
\end{proof}

En vue d'estimer le risque moyen 
\[
\E \Bigl[ \cR_P \bigl( \fh_m (D_n) \bigr) \Bigr]
\, ,
\]
d'apr\`es~\eqref{VC.eq.E.gal}, 
le biais de la validation crois\'ee s'\'ecrit : 
\begin{equation}
\label{VC.eq.biais.gal}
\E \Bigl[ \cR_P \bigl( \fh_m (D_{n_e}) \bigr) \Bigr]
- \E \Bigl[ \cR_P \bigl( \fh_m (D_n) \bigr) \Bigr]
\, . 
\end{equation}
En particulier, il ne d\'epend pas du nombre $V$ de d\'ecoupages ! 
C'est seulement une fonction de la taille $n_e$ de l'\'echantillon d'entra\^inement 
(en plus de $n$, $P$ et $\fh_m$). 
La mani\`ere dont le risque moyen varie avec la taille $n$ de l'\'echantillon 
joue donc un r\^ole cl\'e dans l'analyse en esp\'erance de la validation crois\'ee. 

\medbreak

Supposons tout d'abord que le risque moyen diminue quand on a 
plus d'observations : 
\begin{equation}
\label{VC.eq.regle-intelligente}
\forall P \in \cM_1(\cX \times \cY), 
\qquad 
n \in \N \backslash \{0\} \mapsto 
\E \Bigl[ \cR_P \bigl( \fh_m (D_n) \bigr) \Bigr] 
\text{ est d\'ecroissante,} 
\end{equation} 
en notant $\cM_1(\cX \times \cY)$ l'ensemble 
des mesures de probabilit\'e sur $\cX \times \cY$. 
Alors, le biais \eqref{VC.eq.biais.gal} est une fonction d\'ecroissante 
de la taille $n_e$ de l'\'echantillon d'entra\^inement. 
En particulier, il est minimal lorsque $n_e = n-1$ 
(par exemple, pour le leave-one-out). 
\SAindex{validation croisee@validation crois\'ee!leave-one-out}% 

\begin{remarque}[R\`egles intelligentes]
\label{VC.rk.regle-intelligente}
\SAindex{regle apprentissage@r\`egle d'apprentissage!intelligente}% 
\SAindex{classification supervis\'ee}% 
\SAindex{k plus proches voisins@$k$ plus proches voisins}% 
\SAindex{noyau (r\`egle d'apprentissage)}% 
\SAindex{partition!regle classification@r\`egle de classification}% 
L'hypoth\`ese \eqref{VC.eq.regle-intelligente}, qui semble faible au premier abord, 
est la d\'efinition d'une r\`egle d'apprentissage \og{}intelligente\fg{}\footnote{% 
En anglais, \og{}smart rule\fg{}.} % fin footnote
\citep[section~6.8]{devr_1996}. 
Mais attention ! Toutes les r\`egles d'apprentissage classiques ne sont pas \og{}intelligentes\fg{} : 
par exemple, la r\`egle du plus proche voisin 
et certaines r\`egles par noyau (avec un noyau et une fen\^etre fixes) 
ne sont pas intelligentes en classification binaire 
\citep[section~6.8 et probl\`emes 6.14--6.15]{devr_1996} ! 
Une r\`egle par partition sur une partition $\cA$ ind\'ependante de~$n$ 
n'est pas non plus intelligente, 
m\^eme si elle l'est \og{}presque\fg{} (voir les exercices 
\ref{VC.exo.intelligente-majorite} et \ref{VC.exo.intelligente-part-random}). 
\citet[probl\`eme~6.16]{devr_1996} conjecturent m\^eme 
qu'aucune r\`egle universellement consistante n'est intelligente. 
\SAindex{consistance!universelle}% 
\end{remarque}

\medbreak

Pour quantifier plus pr\'ecis\'ement le biais, 
faisons une hypoth\`ese plus forte, qui implique \eqref{VC.eq.regle-intelligente} : 
\begin{equation}
\label{VC.eq.hyp-risque-n-OLS}
\exists \alpha(m) \in \R, \, 
\beta(m) > 0, \, 
\forall n \geq 1, \qquad 
\E \Bigl[ \cR_P \bigl( \fh_m (D_n) \bigr) \Bigr] 
= \alpha(m)  + \frac{\beta(m)}{n} 
\, . 
\end{equation}
Par exemple, \eqref{VC.eq.hyp-risque-n-OLS} 
est v\'erifi\'ee en estimation de densit\'e 
\SAindex{moindres carres@moindres carr\'es!estimation de densit\'e}% 
par moindres carr\'es\footnote{% 
Ce cadre n'est pas un exemple de probl\`eme de pr\'evision, 
mais on peut tout de m\^eme y d\'efinir la validation crois\'ee, 
voir la parenth\`ese~\ref{VC.rk-cadre-le-plus-general} 
en section~\ref{VC.sec.def.gal}. } ;  % fin footnote
$\alpha(m)$ est alors l'erreur d'approximation 
\SAindex{erreur approximation@erreur d'approximation}% 
et $\beta(m)/n$ est l'erreur d'estimation 
\SAindex{erreur estimation@erreur d'estimation}% 
(deux quantit\'es d\'efinies par \citet[section~3.4]{Arl_2016_JESchap2}). 
\SAindex{partition!regle regression@r\`egle de r\'egression}% 
En r\'egression avec le co\^ut quadratique, 
\eqref{VC.eq.hyp-risque-n-OLS} 
est approximativement v\'erifi\'ee pour les r\`egles par partition 
\citep{arlo_2008}. 
Alors, le biais \eqref{VC.eq.biais.gal} d'un crit\`ere 
par validation crois\'ee s'\'ecrit :  
\[
\beta(m) \parenj{ \frac{1}{n_e} - \frac{1}{n} }
= \parenj{ \frac{n}{n_e} - 1 } \frac{\beta(m)}{n} 
\, . 
\]
On peut distinguer trois situations : 
\begin{itemize}
\item si $n_e \sim n$ (comme pour le leave-one-out), 
\SAindex{validation croisee@validation crois\'ee!leave-one-out}% 
alors le biais est n\'egligeable devant $\beta(m)/n$ : 
au premier ordre, la validation crois\'ee estime 
le risque moyen sans biais. 
\item si $n_e = \kappa n$ avec $\kappa \in \,\, ]0,1[$, 
alors le biais vaut 
$( \kappa^{-1} - 1 ) \beta(m) / n$ : 
la validation crois\'ee estime correctement l'erreur d'approximation 
$\alpha(m)$ mais surestime d'un facteur $\kappa^{-1}>1$ l'erreur d'estimation 
$\beta(m)/n$. 
\SAindex{validation croisee@validation crois\'ee!V-fold@$V$-fold}% 
C'est notamment le cas de la validation crois\'ee $V$-fold 
avec $\kappa^{-1} - 1 = 1/(V-1)$. 
\item si $n_e \ll n$, alors le biais est de l'ordre de 
\[
\frac{\beta(m)}{n_e} \gg \frac{\beta(m)}{n}
\, . 
\]
La validation crois\'ee surestime fortement l'erreur d'estimation,  
et donc aussi le risque (sauf si l'erreur d'approximation $\alpha(m)$ 
domine l'erreur d'estimation, auquel cas il peut n'y avoir quasiment pas de surestimation). 
\end{itemize}

\subsection{Correction du biais}
\label{VC.sec.estim-risque.corr}
\SAindex{validation croisee@validation crois\'ee!corrig\'ee|(}% 

Plut\^ot que de minimiser le biais en choisissant $n_e$ proche de $n$ 
(ce qui n\'ecessite souvent de consid\'erer un grand nombre $V$ 
de d\'ecoupages \`a cause de la variance), 
il est parfois possible de corriger le biais. 
\begin{definition}[Validation crois\'ee corrig\'ee]
\label{VC.def.VC-corrigee}
Soit $\fh_m$ une r\`egle d'apprentissage. 
L'estimateur par validation crois\'ee corrig\'ee\footnote{% 
Le terme anglais est \og{}bias-corrected cross-validation\fg{}.} %fin footnote
du risque de $\fh_m$ pour l'\'echantillon $D_n$ et la suite de d\'ecoupages 
$(E_j)_{1 \leq j \leq V}$ est d\'efini par :  
\begin{align*}
\cRh^{\mathrm{vc-cor}} \bigl( \fh_m; D_n; (E_j)_{1 \leq j \leq V} \bigr) 
&= 
\cRh^{\mathrm{vc}} \bigl( \fh_m; D_n; (E_j)_{1 \leq j \leq V} \bigr) 
\\
&\qquad 
+ \cRh_n \bigl( \fh_m (D_n) \bigr)
- \frac{1}{V} \sum_{j=1}^V \cRh_n \bigl( \fh_m (D_n^{E_j}) \bigr) 
\, . 
\end{align*}
\end{definition}

La validation crois\'ee corrig\'ee a \'et\'e propos\'ee par \citet{burm_1989}, 
qui la justifie par des arguments asymptotiques, 
en supposant que $\fh_m$ est \og{}r\'eguli\`ere\fg{}. 
\`A la suite de \citet{arlo_2008} et \citet{arlo_2016}, 
on peut d\'emontrer qu'elle est \emph{exactement} sans biais, pour tout $n\geq 1$, 
sous une hypoth\`ese similaire \`a \eqref{VC.eq.hyp-risque-n-OLS}. 

\begin{proposition}
\label{VC.pro.VC-corrigee.sans-biais}
On suppose \eqref{VC.eq.hyp-Ind} v\'erifi\'ee et qu'une constante $\gamma(m) \in \R$ existe  
telle que pour tout entier $n\geq 1$ :  
\begin{equation}
\label{VC.eq.hyp-penid-n-OLS}
\E \Bigl[ \cR_P \bigl( \fh_m (D_n) \bigr) 
- \cRh_n \bigl( \fh_m (D_n) \bigr) \Bigr] 
= \frac{\gamma(m)}{n} 
\, . 
\end{equation}
Alors, pour tout $n\geq 1$ et toute suite de d\'ecoupages $(E_j)_{1 \leq j \leq V}$, 
la validation crois\'ee corrig\'ee estime sans biais le risque moyen de $\fh_m$ :  
\begin{equation}
\label{VC.eq.pro.VC-corrigee.sans-biais}
\E \Bigl[ \cRh^{\mathrm{vc-cor}} \bigl( \fh_m; D_n; (E_j)_{1 \leq j \leq V} \bigr)  \Bigr] 
= \E \Bigl[ \cR_P \bigl( \fh_m (D_n) \bigr) \Bigr]
\, . 
\end{equation}
\end{proposition}

\begin{proof}
On part de la d\'efinition de la validation crois\'ee corrig\'ee, 
en remarquant que : 
\[
\cRh_n^{E_j^c} - \cRh_n 
= 
\frac{\card(E_j)}{n} \bigl( \cRh_n^{E_j^c} - \cRh_n^{E_j} \bigr)
\, . 
\]
Alors, 
\begin{align*}
& \qquad 
\cRh^{\mathrm{vc-cor}} \bigl( \fh_m; D_n; (E_j)_{1 \leq j \leq V} \bigr) 
\\
&= 
\cRh_n \bigl( \fh_m (D_n) \bigr)
+ \frac{1}{V} \sum_{j=1}^V 
\Bigl[
( \cRh_n^{E_j^c} - \cRh_n) \bigl( \fh_m (D_n^{E_j}) \bigr) 
\Bigr]
\\
&= 
\cRh_n \bigl( \fh_m (D_n) \bigr)
+ \frac{1}{V} \sum_{j=1}^V 
\frac{\card(E_j)}{n}
\Bigl[
( \cRh_n^{E_j^c} - \cRh_n^{E_j}) \bigl( \fh_m (D_n^{E_j}) \bigr) 
\Bigr]
\\
&= 
\cR_P \bigl( \fh_m (D_n) \bigr)
+ (\cRh_n - \cR_P) \bigl( \fh_m (D_n) \bigr)
\\
&\qquad + \frac{1}{V} \sum_{j=1}^V 
\frac{\card(E_j)}{n}
\Bigl[
(\cRh_n^{E_j^c} - \cR_P) \bigl( \fh_m (D_n^{E_j}) \bigr) 
\Bigr]
\\
& \qquad 
+ \frac{1}{V} \sum_{j=1}^V 
\frac{\card(E_j)}{n}
\Bigl[
(\cR_P - \cRh_n^{E_j}) \bigl( \fh_m (D_n^{E_j}) \bigr) 
\Bigr]
\, . 
\end{align*}
Or, le troisi\`eme terme ci-dessus est d'esp\'erance nulle 
car $D_n^{E_j}$ et $D_n^{E_j^c}$ sont ind\'ependants 
d'apr\`es \eqref{VC.eq.hyp-Ind}. 
En utilisant l'hypoth\`ese \eqref{VC.eq.hyp-penid-n-OLS}, 
on en d\'eduit que 
\begin{align*}
& \qquad 
\E \Bigl[ 
\cRh^{\mathrm{vc-cor}} \bigl( \fh_m; D_n; (E_j)_{1 \leq j \leq V} \bigr) 
- \cR_P \bigl( \fh_m (D_n) \bigr) 
\Bigr]
\\
&=  \frac{- \gamma(m)}{n} 
+ \frac{1}{V} \sum_{j=1}^V 
\frac{\card(E_j)}{n} \frac{\gamma(m)}{\card(E_j)}
= 0
\, . 
\end{align*}
\end{proof}

\begin{rqpar}[Sur l'hypoth\`ese \eqref{VC.eq.hyp-penid-n-OLS}]
\label{VC.rk.eq.hyp-penid-n-OLS} 
%%% parenth\`ese~\ref{VC.rk.eq.hyp-penid-n-OLS} en section~\ref{VC.sec.estim-risque.corr}
\SAindex{selection modeles@s\'election de mod\`eles!p\'enalisation}% 
\SAindex{partition!regle regression@r\`egle de r\'egression}% 
L'hypoth\`ese \eqref{VC.eq.hyp-penid-n-OLS} porte sur l'esp\'erance de la 
p\'enalit\'e id\'eale (voir \citet[section~3.9]{Arl_2016_JESchap2}). 
Elle est v\'erifi\'ee en estimation de densit\'e par moindres carr\'es \citep{arlo_2016}, 
\SAindex{moindres carres@moindres carr\'es!estimation de densit\'e}% 
et elle l'est (approximativement) pour les r\`egles de r\'egression par partition 
avec le co\^ut quadratique \citep{arlo_2008}. 
On peut noter que dans les deux cas, les hypoth\`eses 
\eqref{VC.eq.hyp-risque-n-OLS} et \eqref{VC.eq.hyp-penid-n-OLS} 
sont v\'erifi\'ees avec $\gamma(m) = 2 \beta(m)$, 
ce qui est li\'e \`a l'heuristique de pente 
\citep{birg_2006,arlo_2009b,lera_2012}. 
De mani\`ere plus g\'en\'erale, l'argument asymptotique de \citet{burm_1989} 
peut se r\'esumer \`a \'etablir que \eqref{VC.eq.hyp-penid-n-OLS} est approximativement valide 
sous des hypoth\`eses de r\'egularit\'e sur $\fh_m$ et sur le co\^ut~$c$. 
\end{rqpar}
\SAindex{validation croisee@validation crois\'ee!corrig\'ee|)}% 

\subsection{Variance}
\label{VC.sec.estim-risque.var}

L'\'etude de la variance des estimateurs par validation crois\'ee 
est plus d\'elicate que celle de leur esp\'erance. 
On peut toutefois \'etablir quelques r\'esultats g\'en\'eraux. 
En revanche, des r\'esultats quantitatifs pr\'ecis 
(et valables pour les proc\'edures par blocs) 
ne sont actuellement connus que dans des cadres particuliers, 
comme celui consid\'er\'e \`a la fin de cette section. 

\subsubsection{In\'egalit\'e g\'en\'erale} 

On a tout d'abord une in\'egalit\'e g\'en\'erale entre les variances 
des estimateurs par validation simple, 
par validation crois\'ee et par leave-$p$-out,  
\`a taille d'\'echantillon d'entra\^inement $n_e = n-p$ fix\'ee. 

\begin{proposition} 
%% proposition~\ref{VC.pro.var-ho-geq-lpo} en section~\ref{VC.sec.estim-risque.var}
\label{VC.pro.var-ho-geq-lpo}
\SAindex{validation simple}% 
\SAindex{validation croisee@validation crois\'ee!leave-p-out@leave-$p$-out}% 
Soit $n,V \geq 1$ des entiers, 
$D_n$ un \'echantillon de $n$ variables al\'eatoires ind\'ependantes 
et de m\^eme loi, $n_e \in \{1, \ldots, n-1 \}$ et 
$E_0, E_1, \ldots, E_V$ des parties de $\{1, \ldots, n\}$ 
ind\'ependantes de $D_n$ et 
telles que $\card(E_j) = n_e$ pour tout $j \in \{0, \ldots, V\}$. 
On a alors : 
\begin{align*}
\var\bigl( \cRh^{\mathrm{val}} \bigl( \fh_m; D_n; E_0 ) \bigr) 
&\geq 
\var\Bigl( \cRh^{\mathrm{vc}} \bigl( \fh_m; D_n; (E_j)_{1 \leq j \leq V} \bigr) \Bigr)
\\
&\geq 
\var \bigl( \cRh^{\mathrm{lpo}} ( \fh_m; D_n ; n - n_e ) \bigr)
\, . 
\end{align*}
\end{proposition}
\begin{proof}
On commence par faire une remarque g\'en\'erale.
Si $Z_1, \ldots, Z_K$ sont des variables al\'eatoires r\'eelles de m\^eme loi, alors : 
\begin{equation}
\label{VC.eq.pro.var-ho-geq-lpo.pr.1}
\var(Z_1) \geq \var\parenj{ \frac{1}{K} \sum_{i=1}^K Z_i } 
\, . 
\end{equation}
En effet, par convexit\'e, on a : 
\[ 
\frac{1}{K} \sum_{i=1}^K Z_i^2
\geq 
\mathopen{}\left( \frac{1}{K} \sum_{i=1}^K Z_i \right)^2 \mathclose{}
\, . 
\]
En int\'egrant cette in\'egalit\'e, on obtient : 
\[ 
\E \bigl[ Z_1^2 \bigr]
\geq 
\E\crochj{ \mathopen{}\left( \frac{1}{K} \sum_{i=1}^K Z_i \right)^2 \mathclose{}  }
\, . 
\]
Le r\'esultat s'en d\'eduit en retranchant de chaque c\^ot\'e : 
\[ 
\Bigl( \E \bigl[ Z_1 \bigr] \Bigr)^2 
= 
\mathopen{}\left( \E\crochj{ \frac{1}{K} \sum_{i=1}^K Z_i } \right)^2 \mathclose{} 
\, . 
\]

Or, la loi de 
$\cRh^{\mathrm{val}} \bigl( \fh_m; D_n; E \bigr)$ 
est la m\^eme quel que soit $E \subset \{ 1, \ldots , n \}$ de taille $n_e$,  
y compris $E=E_j$ (par ind\'ependance des $E_j$ avec $D_n$). 
En effet, passer de $E$ \`a $E'$ de m\^eme taille \'equivaut \`a permuter les observations de $D_n$, 
ce qui ne change pas la loi de $D_n$ car les observations sont 
ind\'ependantes et de m\^eme loi. 
L'in\'egalit\'e \eqref{VC.eq.pro.var-ho-geq-lpo.pr.1} implique donc la premi\`ere in\'egalit\'e.

Pour obtenir la deuxi\`eme, notons $\Sigma_n$ l'ensemble des 
permutations de $\{1, \ldots, n\}$ 
et $D_n^{\sigma} = (X_{\sigma(i)},Y_{\sigma(i)})_{1 \leq i \leq n}$ 
pour toute permutation $\sigma \in \Sigma_n$. 
On a alors, pour tout $j \in \{ 1, \ldots, V \}$, 
\begin{align*}
\cRh^{\mathrm{lpo}}\bigl( \fh_m; D_n; n-n_e \bigr)
&= 
\frac{1}{n!} \sum_{\sigma \in \Sigma_n} 
\cRh^{\mathrm{val}} \bigl( \fh_m; D_n^{\sigma}; E_j \bigr) 
\, , 
\\
\text{et donc} \qquad 
\cRh^{\mathrm{lpo}}\bigl( \fh_m; D_n; n-n_e \bigr)
&= 
\frac{1}{n!} \sum_{\sigma \in \Sigma_n} 
\cRh^{\mathrm{vc}} \bigl( \fh_m; D_n^{\sigma}; (E_j)_{1 \leq j \leq V} \bigr)
\, . 
\end{align*}
Or, $D_n^{\sigma}$ a m\^eme loi que $D_n$, puisque 
les observations sont ind\'ependantes et de m\^eme loi, 
si bien que \eqref{VC.eq.pro.var-ho-geq-lpo.pr.1} s'applique 
et donne la deuxi\`eme in\'egalit\'e. 
\end{proof}
\begin{rqpar}[Variance conditionnelle et \eqref{VC.eq.pro.var-ho-geq-lpo.pr.1}]
En utilisant la formule \eqref{VC.eq.varcond} ci-apr\`es 
et le fait qu'une variance est positive ou nulle, on obtient que 
\begin{equation*} 
\var(Z) \geq  \var \bigl( \E [Z \, \vert \, X] \bigr) 
\, ,
\end{equation*}
ce qui d\'emontre l'in\'egalit\'e \eqref{VC.eq.pro.var-ho-geq-lpo.pr.1} 
d'une autre mani\`ere. 
\end{rqpar}

Le r\'esultat de la proposition~\ref{VC.pro.var-ho-geq-lpo} est naturel : 
\`a taille d'\'echantillon d'entra\^inement $n_e$ fix\'ee, 
il vaut mieux consid\'erer plusieurs d\'ecoupages qu'un seul, 
et il vaut encore mieux les consid\'erer tous. 
Mais c'est un r\'esultat limit\'e : 
l'am\'elioration est-elle stricte ? 
Si oui, combien gagne-t-on ? 
La proposition~\ref{VC.pro.var-ho-geq-lpo} ne permet pas de savoir. 
Pire encore : on ne peut pas comparer ce qui se passe avec deux et avec trois d\'ecoupages ! 

Intuitivement, plus le nombre $V$ de d\'ecoupages est grand, 
plus la variance est petite. 
La r\'ealit\'e est malheureusement un peu plus compliqu\'ee, 
et cela d\'epend de \emph{comment} on d\'ecoupe. 
Cinq d\'ecoupages bien choisis peuvent \^etre meilleurs que six mal choisis ! 

\subsubsection{Validation crois\'ee Monte-Carlo}

\SAindex{validation croisee@validation crois\'ee!Monte-Carlo|(}% 
Si l'on se limite \`a une famille particuli\`ere de d\'ecoupages 
(la validation crois\'ee Monte-Carlo), 
cette intuition peut \^etre justifi\'ee et quantifi\'ee. 

\begin{proposition}
%% proposition~\ref{VC.pro.var-MCCV} en section~\ref{VC.sec.estim-risque.var}
\label{VC.pro.var-MCCV}
Soit $n > n_e \geq 1$ des entiers. 
Si $E_1, \ldots, E_V \subset \{1, \ldots, n\}$ 
sont ind\'ependantes, de loi uniforme sur l'ensemble des parties 
de $\{1, \ldots, n\}$ de taille $n_e$, 
et ind\'ependantes de $D_n$, alors on a :  
\begin{align*}
&\qquad 
\var \Bigl( \cRh^{\mathrm{vc}} \bigl( \fh_m ; D_n ; (E_j)_{1 \leq j \leq V} \bigr) \Bigr) 
\\
&= 
\var \bigl( \cRh^{\mathrm{lpo}} ( \fh_m ; D_n ; n-n_e ) \bigr) 
%%\\
%%&\qquad 
+ \frac{1}{V} \E \biggl[ \var \Bigl( \cRh_n^{E_1^c} \bigl( \fh_m \bigl(D_n^{E_1}\bigr) \bigr) 
\, \Big\vert \, D_n \Bigr) \biggr]
\\
&= 
\var \bigl( \cRh^{\mathrm{lpo}} ( \fh_m ; D_n ; n-n_e ) \bigr) 
\\
&\qquad 
+ \frac{1}{V} \Bigl[ 
\var \bigl( \cRh^{\mathrm{val}}( \fh_m ; D_n ; E_1) \bigr)
- \var \bigl( \cRh^{\mathrm{lpo}} ( \fh_m ; D_n ; n-n_e ) \bigr) 
\Bigr]
\, . 
\end{align*}
\end{proposition}

\begin{proof}
Avant toutes choses, rappelons la d\'efinition de la variance conditionnelle 
\SAindex{variance conditionnelle}% 
et une propri\'et\'e fondamentale. 
Si $X$ et $Z$ sont deux variables al\'eatoires,  
si $Z$ est \`a valeurs r\'eelles et si $\E [ Z^2 \, \vert \, X ] < + \infty$, 
la variance conditionnelle de $Z$ sachant $X$ est d\'efinie par : 
\[
\var ( Z \, \vert \, X )
= \E \bigl[ Z^2 \, \vert \, X \bigr] - \E [Z \, \vert \, X]^2
\, . 
\]
On a alors, en int\'egrant cette d\'efinition : 
\begin{equation} 
\label{VC.eq.varcond}
\var(Z) = \E\bigl[ \var ( Z \, \vert \, X) \bigr] 
+ \var \bigl( \E [Z \, \vert \, X] \bigr) 
\, .
\end{equation}

En vue d'appliquer la formule \eqref{VC.eq.varcond}, on pose : 
\[ 
X = D_n \qquad \text{et} \qquad 
Z = \cRh^{\mathrm{vc}} \bigl( \fh_m; D_n; (E_j)_{1 \leq j \leq V} \bigr)
\, . 
\]
D'une part, puisque les $E_j$ sont de loi uniforme sur l'ensemble des parties de 
$\{ 1, \ldots, n\}$ de taille $n_e$, on obtient : 
\[ 
\E[Z \, \vert \, D_n] = \cRh^{\mathrm{lpo}} ( \fh_m ; D_n ; n-n_e )
\, . 
\]
D'autre part, conditionnellement \`a $D_n$, 
les $E_j$ sont ind\'ependantes et de m\^eme loi, si bien que l'on a : 
\[
\var( Z \, \vert \, D_n) = 
\frac{1}{V} \var \biggl( \cRh_n^{E_1^c} \Bigl( \fh_m \bigl(D_n^{E_1}\bigr) \Bigr) 
\, \Big\vert \, D_n \biggr) 
\]
Avec \eqref{VC.eq.varcond}, on en d\'eduit la premi\`ere formule. 

Lorsque $V=1$, la validation crois\'ee Monte-Carlo correspond 
\`a la validation simple, et donc :  
\begin{align*}
&\qquad \var\Bigl( 
\cRh^{\mathrm{val}} \bigl( \fh_m; D_n; E_1 \bigr)
\Bigr) 
\\
& = \var\Bigl( \cRh^{\mathrm{lpo}}\bigl( \fh_m; D_n; n-n_e \bigr) \Bigr)
+ 
\E \biggl[ \var \Bigl( \cRh_n^{E_1^c} \bigl( \fh_m \bigl(D_n^{E_1}\bigr) \bigr) 
\, \Big\vert \, D_n \Bigr) \biggr]
\, . 
\end{align*}
La deuxi\`eme formule en d\'ecoule. 
\end{proof}

La proposition~\ref{VC.pro.var-MCCV} d\'emontre que la variance du crit\`ere 
par validation crois\'ee Monte-Carlo 
est une fonction d\'ecroissante du nombre $V$ de d\'ecoupages, 
ce qui confirme l'intuition \'enonc\'ee pr\'ec\'edemment. 
Plus pr\'ecis\'ement, cette variance est une fonction affine de $1/V$, 
comprise entre la variance du crit\`ere par validation simple 
($V=1$) 
et la variance du crit\`ere leave-$p$-out correspondant 
(obtenu quand $V \rightarrow +\infty$). 
En particulier, lorsque les variances des crit\`eres 
par validation simple et par leave-$p$-out 
ne diff\`erent que par une constante multiplicative, 
l'essentiel de l'effet de $V$ sur la d\'ecroissance de la variance 
est obtenu avec un nombre $V$ de d\'ecoupages petit  
(par exemple, $V=20$ suffit \`a obtenir au moins $95\%$ de la 
d\'ecroissance esp\'er\'ee).  
\SAindex{validation croisee@validation crois\'ee!Monte-Carlo|)}% 

\begin{remarque}
\label{VC.rk.var.pro-ho-lpo-et-MCCV.generalisation}
\SAindex{validation croisee@validation crois\'ee!corrig\'ee}% 
Les propositions \ref{VC.pro.var-ho-geq-lpo} et \ref{VC.pro.var-MCCV} 
s'\'etendent \`a la validation crois\'ee corrig\'ee, 
et plus g\'en\'eralement \`a toute quantit\'e de la forme 
\[
\frac{1}{V} \sum_{j=1}^V F (D_n; E_j)
\]
avec $E_1, \ldots, E_V$ ind\'ependantes, de m\^eme loi et ind\'ependantes de $D_n$. 
Elles s'appliquent donc aussi \`a la 
\SAindex{validation croisee@validation crois\'ee!V-fold repetee@$V$-fold r\'ep\'et\'ee}% 
validation crois\'ee $V$-fold r\'ep\'et\'ee, 
d\'efinie \`a la fin de la section~\ref{VC.sec.def.ex}  
(voir l'exercice~\ref{VC.exo.var-VF-repete}). 
\end{remarque}

\medbreak

Il est malheureusement difficile d'\'enoncer un r\'esultat plus pr\'ecis en toute g\'en\'eralit\'e. 
On peut seulement d\'etailler comme suit la variance de l'estimateur 
par validation simple : 
\SAindex{validation simple}% 
\begin{equation}
\label{VC.eq.var-hold-out}
\begin{split}
\var \Bigl( \cRh_n^{E^c} \bigl( \fh_m (D_n^E) \bigr) \Bigr) 
&= 
\var \Bigl( \cR_P \bigl( \fh_m (D_n^E) \bigr) \Bigr)
\\
&\hspace{-1cm}+ 
\frac{1}{\card(E^c)} \E \biggl[ \var \Bigl( c \bigl( \fh_m ( D_n^E ; X) , Y \bigr) 
\, \Big\vert \, \fh_m (D_n^E) \Bigr) \biggr]
\, . 
\end{split}
\end{equation}
La formule ci-dessus souligne le r\^ole de l'\'echantillon de validation, 
qui fait diminuer la variance de l'estimation du risque de $\fh_m(D_n^E)$, 
jusqu'\`a atteindre (virtuellement) sa valeur minimale possible 
(la variance du risque de $\fh_m(D_n^E)$) 
lorsque l'\'echantillon de validation est infini. 
Soulignons n\'eanmoins l'int\'er\^et limit\'e de \eqref{VC.eq.var-hold-out}, 
qui ne permet pas (telle quelle) 
de comparer diff\'erentes valeurs de $n_e = \card(E)$, 
notamment parce que $n_e$ est li\'e \`a 
la taille de l'\'echantillon de validation $\card(E^c) = n-n_e$. 

\subsubsection{Quantification pr\'ecise dans un cadre particulier} 

\SAindex{moindres carres@moindres carr\'es!estimation de densit\'e|(}% 
En consid\'erant une r\`egle $\fh_m$ et un co\^ut $c$ particuliers, 
il est parfois possible d'\'enoncer des r\'esultats plus pr\'ecis, 
soit asymptotiquement (quand $n$ tend vers l'infini, avec 
$n_e = \card(E_j)$ \'eventuellement li\'e \`a $n$), 
soit pour une valeur de $n$ fix\'ee 
\citep[section~5.2]{arlo_2010}. 
\`A titre d'illustration, 
consid\'erons le cas de l'estimation de densit\'e 
avec le co\^ut des moindres carr\'es 
et $\fh_m$ un estimateur par histogrammes r\'eguliers 
de pas $h_m>0$ 
\citep[\'enoncent un r\'esultat valable pour 
des estimateurs par projection g\'en\'eraux]{arlo_2016}. 
On a alors, pour la validation crois\'ee $V$-fold :  
\SAindex{validation croisee@validation crois\'ee!V-fold@$V$-fold}% 
\begin{equation}
\label{VC.var-VFCV.histo-reg-densite}
\begin{split}
&\var \Bigl( \cRh^{\mathrm{vf}} \bigl( \fh_m ; D_n ; (B_j)_{1 \leq j \leq V} \bigr) \Bigr)
\\
&\qquad = \frac{1}{n^2} C_1^{\mathrm{vf}} (V,n) \cW_1 ( h_m, P)
+ \frac{1}{n} C_2^{\mathrm{vf}} (V,n) \cW_2 ( h_m, P)
\end{split}
\end{equation} 
o\`u les $\cW_i(h_m,P)$ ne d\'ependent que de $h_m$ et de la loi $P$ des observations 
\citep[en donnent une formule close en section~5]{arlo_2016}, 
\begin{align*}
C_1^{\mathrm{vf}} (V,n)
&= 
1 + \frac{4}{V-1} + \frac{4}{(V-1)^2} + \frac{1}{(V-1)^3} - \frac{V^2}{n (V-1)^2}
\\
\text{et} \qquad 
C_2^{\mathrm{vf}} (V,n)
&= 
\mathopen{} \left( 1 + \frac{V}{n(V-1)} \right)^2 \mathclose{}
\, . 
\end{align*}
En particulier, la variance diminue avec $V$ 
et le gain induit par l'augmentation de $V$ est limit\'e : 
pour $n$ assez grand, 
le premier terme varie de 10 \`a 1 quand $V$ augmente,  
tandis que le deuxi\`eme terme reste \`a peu pr\`es constant. 
Lorsque $h_m$ est fixe par rapport \`a $n$, 
la diminution de variance li\'ee \`a l'augmentation de $V$ est encore plus faible : 
elle ne se voit que dans le premier terme du membre de droite 
de \eqref{VC.var-VFCV.histo-reg-densite}, qui est un terme du deuxi\`eme ordre. 

\SAindex{validation croisee@validation crois\'ee!Monte-Carlo|(}%  
Une formule similaire peut \^etre \'etablie dans le m\^eme cadre pour la validation 
crois\'ee Monte-Carlo \citep[section~8.1]{arlo_2016} : 
\begin{equation}
\label{VC.var-MCCV.histo-reg-densite}
\begin{split}
& \var \Bigl( \cRh^{\mathrm{vc}} \bigl( \fh_m ; D_n ; (E_j)_{1 \leq j \leq V} \bigr) \Bigr)
\\
&\qquad 
= \frac{1}{n^2} C_1^{\mathrm{MC}} (V, n, n_e) \cW_1 ( h_m, P)
+ \frac{1}{n} C_2^{\mathrm{MC}} (V, n, n_e) \cW_2 ( h_m, P)
\end{split}
\end{equation} 
o\`u les $\cW_i$ sont les m\^emes qu'\`a l'\'equation \eqref{VC.var-VFCV.histo-reg-densite}, 
\begin{align*}
C_1^{\mathrm{MC}} (V,n,n_e)
&= 
\frac{1}{V} \parenj{ \frac{n^2}{n_e^2} + \frac{2 n^2}{n_e (n-n_e)} - \frac{1}{n} \frac{n^3}{n_e^3}  }
\\
&\qquad + \parenj{ 1 - \frac{1}{V} } \crochj{ 1 + \frac{1}{n-1} 
    \mathopen{} \left( \frac{n}{n_e} + 1 \right)^2 \mathclose{} - \frac{1}{n} \frac{n^2}{n_e^2}  }
\\
\text{et} \qquad 
C_2^{\mathrm{MC}} (V,n,n_e)
&= 
\frac{1}{V} \parenj{ \frac{n}{n-n_e} + \frac{1}{n^2} \frac{n^3}{n_e^3} }
+ \parenj{ 1 - \frac{1}{V} } \mathopen{} \left( 1 + \frac{1}{n} \frac{n}{n_e} \right)^2 \mathclose{}
\, . 
\end{align*}
Ainsi, lorsque $n$ tend vers l'infini et $n_e \sim \tau n$ avec $\tau \in \,\, ]0,1[$ fix\'e, 
les valeurs extr\^emes de la variance sont donn\'ees par 
\begin{align*}
C_1^{\mathrm{MC}} (1,n,n_e)
&\xrightarrow[n \rightarrow +\infty]{} \frac{1}{\tau^2} + \frac{2}{\tau (1-\tau)} > 11
\\
\text{et} \qquad 
C_2^{\mathrm{MC}} (1,n,n_e)
&\xrightarrow[n \rightarrow +\infty]{} \frac{1}{1-\tau} > 1
\end{align*}
\SAindex{validation simple}% 
d'une part (pour la validation simple), et par  
\begin{align*}
C_1^{\mathrm{MC}} (\infty,n,n_e)
\xrightarrow[n \rightarrow +\infty]{} 1
\qquad \text{et} \qquad 
C_2^{\mathrm{MC}} (\infty,n,n_e)
\xrightarrow[n \rightarrow +\infty]{} 1 
\end{align*}
\SAindex{validation croisee@validation crois\'ee!leave-p-out@leave-$p$-out}% 
d'autre part (pour le leave-$(n-n_e)$-out). 
Le gain li\'e \`a l'augmentation du nombre de d\'ecoupages est donc 
de l'ordre d'une constante (fonction de $\tau$ uniquement), 
et il n'est pas \'enorme 
(au plus un facteur $12$ lorsque $\tau=1/2$, par exemple). 

\SAindex{validation croisee@validation crois\'ee!V-fold@$V$-fold}% 
Les formules \eqref{VC.var-VFCV.histo-reg-densite} et \eqref{VC.var-MCCV.histo-reg-densite} 
permettent \'egalement d'\'evaluer l'int\'er\^et de consid\'erer une suite 
\og{}structur\'ee\fg{} de d\'ecoupages (avec le $V$-fold) 
par rapport \`a une suite al\'eatoire (avec la m\'ethode Monte-Carlo). 
Il suffit pour cela de comparer les valeurs des constantes 
$C_i$ pour une m\^eme taille d'\'echantillon d'entra\^inement 
$n_e = n (V-1)/V$ et un m\^eme nombre de d\'ecoupages $V$.  
Lorsque $n$ tend vers l'infini, 
\[ 
\forall V \geq 3, 
\qquad 
\frac{C_1^{\mathrm{MC}} \parenj{  V, n , \frac{n (V-1)}{V} }}{ C_1^{\mathrm{VF}} (V, n ) }
> 1
\]
et ce rapport tend vers $3$ lorsque $V$ tend vers l'infini, 
tandis que 
\[
\frac{C_2^{\mathrm{MC}} \parenj{  V, n , \frac{n (V-1)}{V} }}{ C_2^{\mathrm{VF}} (V, n ) }
\xrightarrow[n \rightarrow +\infty]{} 2 - \frac{1}{V} \in \crochj{ \frac{3}{2} , 2}
\, . 
\]
La validation crois\'ee par blocs permet donc d'avoir 
une variance plus faible qu'avec la validation crois\'ee Monte-Carlo, 
avec un gain de l'ordre d'une constante num\'erique, 
compris entre $3/2$ et $3$ environ\footnote{% 
Le gain exact d\'epend de l'ordre de grandeur relatif de 
$\cW_1(h_m,P)$ et $\cW_2(h_m,P)$. }.  

\begin{rqpar}[Validation crois\'ee corrig\'ee]
%\`u\label{VC.rk.var.VCcorr}
\SAindex{validation croisee@validation crois\'ee!corrig\'ee}%  
La variance du crit\`ere par validation crois\'ee corrig\'ee 
s'\'ecrit d'une mani\`ere similaire \`a \eqref{VC.var-VFCV.histo-reg-densite} 
(pour le $V$-fold) 
et \eqref{VC.var-MCCV.histo-reg-densite} 
(pour la proc\'edure Monte-Carlo), 
avec les m\^emes conclusions sur l'influence de~$V$. 
\end{rqpar}
\SAindex{moindres carres@moindres carr\'es!estimation de densit\'e|)}% 
\SAindex{validation croisee@validation crois\'ee!Monte-Carlo|)}%  

\medbreak

Notons pour finir que l'essentiel des conclusions obtenues ci-dessus 
pour les histogrammes en estimation de densit\'e sont 
probablement valables pour tout pr\'edicteur $\fh_m$ \og{}r\'egulier\fg{}, 
au moins asymptotiquement au vu des r\'esultats de \citet{burm_1989}.  
Toutefois, des comportements exp\'erimentaux diff\'erents sont rapport\'es 
dans la litt\'erature, 
notamment pour des pr\'edicteurs \og{}instables\fg{}, 
\SAindex{stabilit\'e}% 
\SAindex{arbre de decision@arbre de d\'ecision!CART}% 
\SAindex{reseau de neurones@r\'eseau de neurones}% 
\SAindex{selection modeles@s\'election de mod\`eles!p\'enalisation}% 
par exemple un arbre CART, 
un r\'eseau de neurones 
ou le r\'esultat d'une proc\'edure 
de s\'election de variables par p\'enalisation \og{}$\ell^0$\fg{} 
telle que AIC ou $C_p$ \citep{brei_1996}. 
\SAindex{AIC}% 
\SAindex{cp@$C_p$}% 
\citet[section~5.2]{arlo_2010} donnent d'autres r\'ef\'erences \`a ce sujet. 
\SAindex{risque!estimation@estimation d'un risque|)}% 

\section{Propri\'et\'es pour la s\'election d'estimateurs}
\label{VC.sec.sel-estim}
\SAindex{selection estimateurs@s\'election d'estimateurs|(}% 
Revenons au probl\`eme (d\'ecrit en section~\ref{VC.sec.pb}) 
de choisir parmi une famille de r\`egles d'apprentissage 
$(\fh_m)_{m \in \cM_n}$, avec un objectif de pr\'evision : 
on veut minimiser le risque du pr\'edicteur final 
$\fh_{\mh(D_n)} (D_n)$. 
\'Etant donn\'e une suite de d\'ecoupages $(E_j)_{1 \leq j \leq V}$ --- dont 
on suppose toujours qu'elle v\'erifie \eqref{VC.eq.hyp-Ind} et \eqref{VC.eq.hyp-Reg} --- 
la proc\'edure de validation crois\'ee correspondante s\'electionne :  
\[
\mh^{\mathrm{vc}} \bigl( D_n ; (E_j)_{1 \leq j \leq V} \bigr) 
\in \argmin_{m \in \cM_n} \Bigl\{ 
\cRh^{\mathrm{vc}} \bigl( \fh_m ; D_n ; (E_j)_{1 \leq j \leq V} \bigr) \Bigr\}
\, . 
\]

\subsection{Diff\'erence entre s\'election d'estimateurs et estimation du risque}
\label{VC.sec.sel-estim.diff-avec-estim-risq}
\SAindex{risque!estimation@estimation d'un risque}% 
La section~\ref{VC.sec.estim-risque} d\'etaille les performances 
de $\cRh^{\mathrm{vc}} ( \fh_m )$ comme estimateur du risque de $\fh_m$.  
Intuitivement, ceci devrait permettre d'\'evaluer les performances de 
la validation crois\'ee pour la s\'election d'estimateurs. 
Cette intuition est en partie correcte, mais en partie seulement, attention ! 
On constate en effet empiriquement que la meilleure proc\'edure de 
s\'election d'estimateurs ne correspond pas toujours au meilleur estimateur du risque 
\citep{brei_1992}. 
Il y a au moins deux raisons \`a ce ph\'enom\`ene. 

\subsubsection{R\^ole des incr\'ements}
Un premier \'el\'ement d'explication peut \^etre obtenu en consid\'erant 
l'exemple jouet suivant. 
Supposons que l'on veut comparer les proc\'edures 
\[
\mh_1 \in \argmin_{m \in \cM_n} \bigl\{ \cC_1(m;D_n) \bigr\}
\qquad \text{et} \qquad 
\mh_2 \in \argmin_{m \in \cM_n} \bigl\{ \cC_2(m;D_n) \bigr\}
\]
o\`u $\cC_1$ est un estimateur sans biais du risque 
et, pour tout $m \in \cM_n$, 
\[
\cC_2(m;D_n) = \cC_1(m;D_n) + Z
\]
avec $Z$ une variable al\'eatoire de loi normale 
$\cN(1,1)$ ind\'ependante de $D_n$ et de~$m$. 
Clairement, $\mh_1 = \mh_2$ et il n'y a aucune diff\'erence entre 
ces deux proc\'edures pour la s\'election d'estimateurs. 
En revanche, $\cC_2$ est bien plus mauvais que $\cC_1$ 
pour l'estimation du risque : 
$\cC_1$ est un estimateur sans biais du risque 
tandis que $\cC_2$ poss\`ede un biais strictement positif (\'egal \`a~$1$) 
et une variance strictement sup\'erieure \`a celle de $\cC_1$.
Consid\'erer seulement l'esp\'erance et la variance de $\cC_i(m)$ 
pour chaque $m \in \cM_n$ peut donc \^etre trompeur pour 
la s\'election d'estimateurs. 

Que faire pour corriger ce probl\`eme ? 
\`A la suite de \citet[section~4]{arlo_2016}, 
on peut remarquer que l'essentiel est de \emph{bien classer} 
les r\`egles $\fh_m$ les unes par rapport aux autres\footnote{%
\`A premi\`ere vue, l'essentiel est de bien classer la meilleure r\`egle 
parmi $(\fh_m)_{m \in \cM_n}$ par rapport aux autres. 
Mais si l'on demande que ceci soit encore vrai pour toute 
sous-collection $(\fh_m)_{m \in \cM_n'}$, $\cM_n' \subset \cM_n$, 
alors on demande \`a bien classer \emph{toutes} les r\`egles 
$\fh_m$ les unes par rapport aux autres.}. % fin footnote 
Autrement dit, une proc\'edure de la forme 
\[
\mh \in \argmin_{m \in \cM_n} \bigl\{ \cC(m; D_n) \bigr\}
\]
est bonne lorsque, 
avec grande probabilit\'e, 
pour tout $m,m' \in \cM_n$, 
\[ 
\signe \bigl( \cC(m;D_n) - \cC(m';D_n) \bigr) 
= \signe \bigl( \cR_P(\fh_m(D_n)) - \cR_P(\fh_{m'}(D_n) \bigr) 
\, , 
\]
sauf \'eventuellement lorsque $\fh_m$ et $\fh_{m'}$ 
ont des risques \og{}proches\fg{} 
(une erreur entre $m$ et $m'$ \'etant alors de faible importance pour le risque 
du pr\'edicteur final).  
Si l'on veut se focaliser sur une esp\'erance ou une variance 
(pour \'evaluer la qualit\'e d'un crit\`ere $\cC$ ou pour comparer 
deux crit\`eres $\cC_1$ et $\cC_2$), 
c'est donc 
\[ 
\var \bigl( \cC(m; D_n) - \cC(m';D_n) \bigr) 
\, , \qquad 
m,m' \in \cM_n \, , 
\]
qu'il faut consid\'erer ; 
la section~\ref{VC.sec.sel-estim.var} le fait pour le cas de la validation crois\'ee.

\subsubsection{Surp\'enalisation}
\SAindex{selection modeles@s\'election de mod\`eles!p\'enalisation|(}%  
Le paragraphe pr\'ec\'edent explique pourquoi, 
entre deux crit\`eres dont les incr\'ements ont la m\^eme esp\'erance, 
il vaut mieux choisir celui dont la variance des incr\'ements 
est la plus petite. 
Il est moins \'evident de comparer des crit\`eres qui n'ont pas la 
m\^eme esp\'erance. 

Pour ce faire, prenons le point de vue de la p\'enalisation 
\citep[section~3.9]{Arl_2016_JESchap2}, 
en posant\footnote{% 
En d\'efinissant $\penal(m;D_n) = \cC(m;D_n)  - \cRh_n(\fh_m(D_n))$, 
on peut \'ecrire tout crit\`ere $\cC$ comme un crit\`ere empirique p\'enalis\'e. } : % fin footnote 
\[ 
\cC(m;D_n) 
= \cRh_n \bigl( \fh_m(D_n) \bigr) + \penal(m;D_n)
\, . 
\]
L'id\'eal est alors que la p\'enalit\'e $\penal(m;D_n)$ 
soit proche de 
\[ 
\penal_{\mathrm{id}}(m;D_n) := \cR_P \bigl( \fh_m(D_n) \bigr) 
- \cRh_n \bigl( \fh_m(D_n) \bigr) 
\]
ou son esp\'erance (qui est, dans de nombreux cas, proportionnelle \`a la 
\og{}complexit\'e\fg{} de $\fh_m$). 
Consid\'erons alors, pour tout $C > 0$, 
la proc\'edure (th\'eorique) qui s\'electionne 
\[ 
m_C \in \argmin_{m \in \cM_n} \Bigl\{ \cRh_n \bigl( \fh_m(D_n) \bigr) 
+ C \E\bigl[ \penal_{\mathrm{id}}(m;D_n) \bigr] \Bigr\}
\, . 
\]
Lorsque $C=1$, $m_C$ suit le principe d'estimation sans biais du risque. 
\SAindex{selection estimateurs@s\'election d'estimateurs!principe d'estimation sans biais du risque}% 
Lorsque $C>1$, $m_C$ surp\'enalise, c'est-\`a-dire qu'elle choisit une r\`egle d'apprentissage 
moins \og{}complexe\fg{}. 
Lorsque $C<1$, $m_C$ sous-p\'enalise et choisit une r\`egle plus \og{}complexe\fg{}. 
La figure~\ref{VC.fig.surpen} repr\'esente, sur un exemple, la performance pour la s\'election 
d'estimateurs de $m_C$ en fonction de la constante de surp\'enalisation~$C$. 
\begin{figure}[!ht]
\centering
\includegraphics[width=0.8\textwidth]{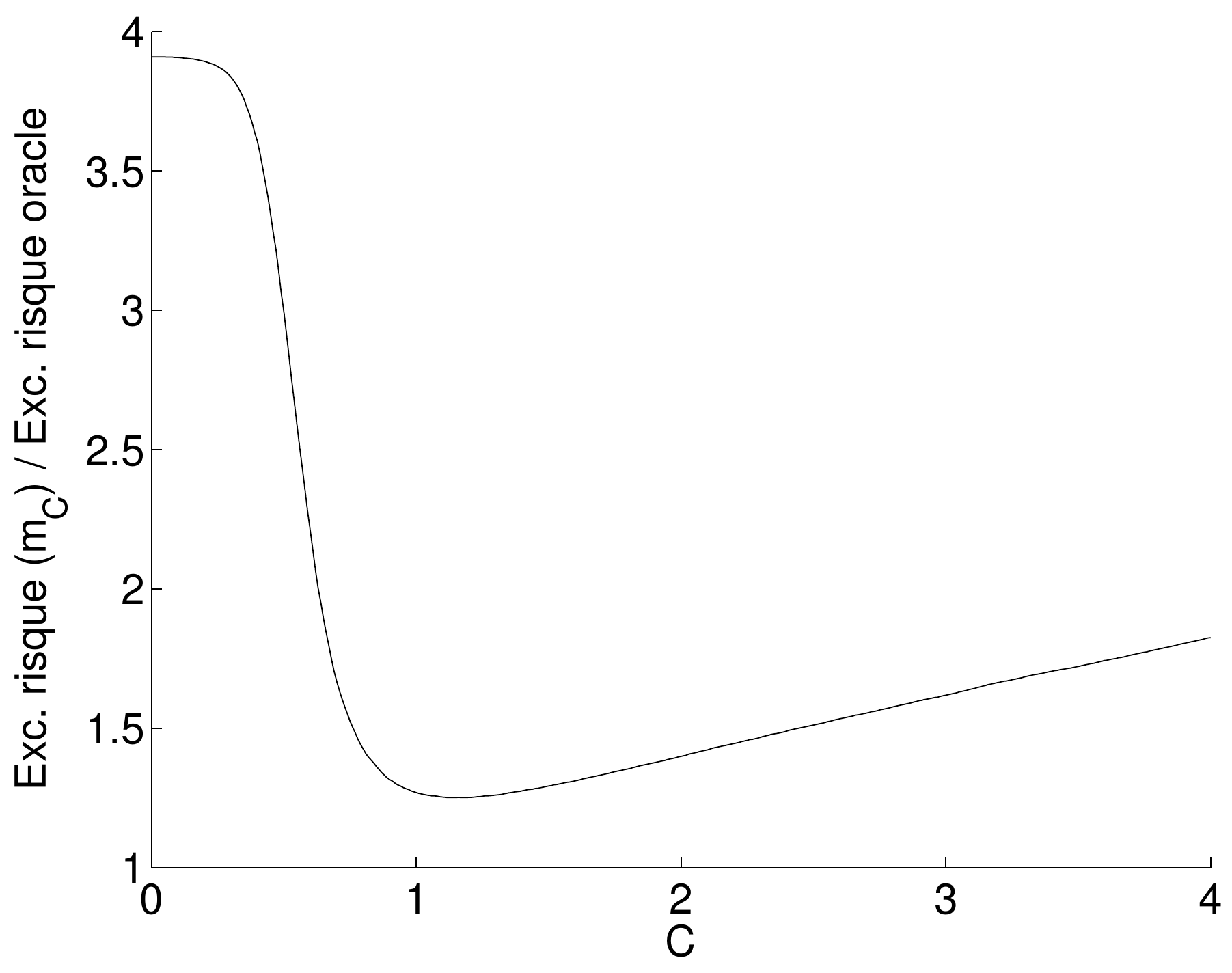}
\caption{\textit{Performance de $m_C$ pour la s\'election de mod\`eles 
\textup{(}mesur\'ee par l'esp\'erance de 
$\ell(\fst,\fh_{m_C})/\inf_{m \in \cM_n} \{ \ell(\fst,\fh_m) \}$\textup{)} 
en fonction de la constante de surp\'enalisation~$C$, 
\'evalu\'ee sur un jeu de donn\'ees simul\'e, 
en r\'egression lin\'eaire avec des estimateurs des moindres carr\'es 
\SAindex{moindres carres@moindres carr\'es!regression lineaire@r\'egression lin\'eaire}% 
\textup{(}m\^eme cadre que celui de \citep[figure~4]{Arl_2016_JESchap2}\textup{)}. 
Ici, la valeur optimale de $C$ est $1,12$. 
}}
\label{VC.fig.surpen}
\end{figure}
On peut observer que la performance est bonne pour $C=1$, 
et qu'elle est un peu meilleure lorsque $C$ est l\'eg\`erement plus grand. 
Surp\'enaliser (donc, utiliser un estimateur \emph{biais\'e} du risque) 
peut ainsi \^etre b\'en\'efique ! 

\SAindex{surapprentissage}% 
\`A notre connaissance, il n'existe pas de r\'esultat th\'eorique 
rendant pleinement compte de ce ph\'enom\`ene. 
On peut n\'eanmoins proposer l'hypoth\`ese suivante : 
en surp\'enalisant, on diminue la probabilit\'e\footnote{% 
Cette probabilit\'e reste non-nulle, \`a cause de la variabilit\'e 
du crit\`ere $\cC(m;D_n)$.} % fin footnote
de surapprendre fortement, 
au prix d'une diminution de la \og{}complexit\'e\fg{} 
moyenne de la r\`egle s\'electionn\'ee. 
Lorsque le gain li\'e au premier point est plus fort que la perte 
li\'ee au deuxi\`eme point, 
il est b\'en\'efique de surp\'enaliser. 

\SAindex{regle du 1 ecart-type@r\`egle du 1 \'ecart-type}% 
La r\`egle du \og{}1 \'ecart-type\fg{}\footnote{% 
En anglais, \og{}1 s.e. rule\fg{}.} de \citet{Breiman84} 
semble \^etre, empiriquement, un bon moyen de surp\'enaliser\footnote{% 
Pr\'ecisons toutefois que cette r\`egle n'est pas directement formul\'ee 
en ces termes. }.

La section~\ref{VC.sec.sel-estim.surpen} discute les 
cons\'equences de ce ph\'enom\`ene sur les proc\'edures de 
validation crois\'ee. 
\SAindex{selection modeles@s\'election de mod\`eles!p\'enalisation|)}%  

\subsection{Analyse au premier ordre : esp\'erance}
\label{VC.sec.sel-estim.E}
Au premier ordre, pour les proc\'edures de validation crois\'ee 
usuelles\footnote{% 
Les proc\'edures utilisant un \'echantillon 
de validation de taille $n-n_e \ll n$ \emph{et} un nombre de d\'ecoupages $V$ petit 
font exception. 
Les garanties th\'eoriques ne s'appliquent pas 
\citep{arlo_2010}, et l'on constate en pratique que ces 
proc\'edures sont extr\^emement variables et se comportent mal, 
m\^eme lorsque leur esp\'erance fournit une bonne proc\'edure 
de s\'election d'estimateurs.}, % fin footnote
les r\'esultats th\'eoriques connus \`a ce jour indiquent tous 
que la performance pour la s\'election d'estimateurs 
est celle de la proc\'edure (th\'eorique) qui minimise l'esp\'erance du crit\`ere correspondant. 

Pour prendre un exemple, 
supposons que les hypoth\`eses \eqref{VC.eq.hyp-Ind},  
\eqref{VC.eq.hyp-Reg} et~\eqref{VC.eq.hyp-risque-n-OLS} 
sont v\'erifi\'ees. 
Alors, pour tout $m,m' \in \cM_n$, on a : 
\begin{align}
\label{VC.eq.E-incr-Rcv-hyp-OLS}
\begin{split}
%\notag 
& 
\E \Bigl[ \cRh^{\mathrm{vc}} \bigl( \fh_m ; D_n ; (E_j)_{1 \leq j \leq V} \bigr) 
- \cRh^{\mathrm{vc}} \bigl( \fh_{m'} ; D_n ; (E_j)_{1 \leq j \leq V} \bigr) \Bigr]
\\
&\qquad = \alpha(m) - \alpha(m') + \frac{\beta(m) - \beta(m')}{n_e}
\end{split}
\\
\label{VC.eq.E-incr-R-hyp-OLS}
\begin{split}
\text{et} \qquad &
\E \Bigl[ \cR_P \bigl( \fh_m (D_n) \bigr) 
- \cR_P \bigl( \fh_{m'} (D_n) \bigr)\Bigr] 
\\
&\qquad 
= \alpha(m) - \alpha(m') + \frac{\beta(m) - \beta(m')}{n}
\, .
\end{split}
\end{align}
Comme en section~\ref{VC.sec.estim-risque.biais}, 
on voit que le biais de la validation crois\'ee vient du 
fait que l'\'echantillon d'entra\^inement est de taille $n_e$ 
au lieu de~$n$ ; 
en particulier, ce biais d\'epend uniquement de $n_e$. 

Quel impact sur la s\'election d'estimateurs ? 
En comparant \eqref{VC.eq.E-incr-Rcv-hyp-OLS} 
et \eqref{VC.eq.E-incr-R-hyp-OLS}, on observe que la diff\'erence 
des erreurs d'approximation $\alpha(m) - \alpha(m')$ 
est correctement estim\'ee, 
tandis que la diff\'erence des erreurs d'estimation 
$(\beta(m) - \beta(m'))/n$ 
est estim\'ee \`a un facteur multiplicatif $n/n_e>1$ pr\`es. 
Puisque $\beta(m)$ mesure en g\'en\'eral la \og{}complexit\'e\fg{} de $\fh_m$ 
(par exemple, via la dimension du mod\`ele sous-jacent), 
ceci signifie que la validation crois\'ee tend \`a s\'electionner une r\`egle $\fh_{\mh}$ 
de complexit\'e plus petite que celle de la meilleure r\`egle (l'oracle). 
C'est naturel : la validation crois\'ee \og{}fait comme si\fg{} chaque r\`egle d'apprentissage 
$\fh_m$ \'etait entra\^in\'ee avec $n_e$ observations, 
ce qui la conduit \`a faire un choix plus conservateur que celui qu'il faut faire 
quand on dispose de $n > n_e$ observations. 

Sur le plan quantitatif, l'impact de $n_e$ est 
en g\'en\'eral similaire \`a ce que l'on a d\'ecrit en section~\ref{VC.sec.estim-risque.biais} 
pour l'estimation du risque. 
\citet{shao_1997} \'enonce des r\'esultats pr\'ecis en r\'egression lin\'eaire par moindres carr\'es 
\SAindex{moindres carres@moindres carr\'es!regression lineaire@r\'egression lin\'eaire}% 
et \citet[section~6]{arlo_2010} donnent les r\'ef\'erences 
de nombreux autres r\'esultats, que l'on peut r\'esumer ainsi. 
Pour la validation crois\'ee, trois cas sont \`a distinguer : 
\begin{itemize}
\item 
lorsque $n_e \sim n$, on obtient une performance 
optimale au premier ordre. 
\item 
lorsque $n_e \sim \kappa n$ avec $\kappa \in \,\, ]0,1[$, 
on obtient une performance sous-optimale : 
le pr\'edicteur final perd un facteur constant 
(fonction de $\kappa$ notamment) par rapport au risque 
du meilleur pr\'edicteur. 
\item 
lorsque $n_e \ll n$, la performance est \'egalement sous-optimale, 
mais la perte est d'un facteur qui tend vers l'infini lorsque 
$n$ tend vers l'infini. 
\end{itemize}
\SAindex{validation croisee@validation crois\'ee!corrig\'ee}%  
Pour la validation crois\'ee corrig\'ee, le biais \'etant nul, 
on obtient une performance optimale au premier ordre 
quel que soit $n_e$. 

\medbreak

La principale limite de ces r\'esultats au premier ordre est qu'ils ne font aucune 
diff\'erence entre des proc\'edures utilisant des \'echantillons 
d'entra\^inement de m\^eme taille et dont les performances empiriques 
sont bien diff\'erentes 
(par exemple, la validation simple et le leave-$p$-out). 
Il semble donc bien qu'en pratique, 
les termes de deuxi\`eme ordre comptent !

\subsection{Analyse au deuxi\`eme ordre : variance}
\label{VC.sec.sel-estim.var}

L'\'etude de la variance (des incr\'ements du crit\`ere, 
d'apr\`es la section~\ref{VC.sec.sel-estim.diff-avec-estim-risq}) 
permet d'expliquer bon nombre de ph\'enom\`enes observ\'es empiriquement 
(en supposant l'heuristique de la section~\ref{VC.sec.sel-estim.diff-avec-estim-risq} 
correcte, ce que l'on fait tout au long de cette section). 

Tout d'abord, les r\'esultats g\'en\'eraux de la section~\ref{VC.sec.estim-risque.var},  
sur la variance d'un estimateur par validation crois\'ee 
du risque d'une r\`egle $\fh_m$, 
se g\'en\'eralisent \`a la variance des incr\'ements 
\[
 \cRh^{\mathrm{vc}} \bigl( \fh_m; D_n; (E_j)_{1 \leq j \leq V} \bigr) 
 - \cRh^{\mathrm{vc}} \bigl( \fh_{m'}; D_n; (E_j)_{1 \leq j \leq V} \bigr)
\]
quelles que soient $\fh_m$ et $\fh_{m'}$ deux r\`egles d'apprentissage. 
Ceci confirme donc l'intuition selon laquelle, 
\`a taille d'\'echantillon d'entra\^inement $n_e$ fix\'ee, 
le leave-$(n-n_e)$-out est la meilleure proc\'edure de s\'election d'estimateurs, 
\SAindex{validation croisee@validation crois\'ee!leave-p-out@leave-$p$-out}% 
\SAindex{validation simple}% 
la validation simple la moins bonne, 
toutes les autres proc\'edures de validation crois\'ee ayant des performances entre les deux. 
Et pour la validation crois\'ee Monte-Carlo, 
\SAindex{validation croisee@validation crois\'ee!Monte-Carlo}% 
la performance en s\'election d'estimateurs s'am\'eliore quand $V$ augmente. 
De plus, la variance des incr\'ements \'etant une fonction affine croissante de $1/V$, 
si les valeurs de variance sont comparables pour validation simple et pour 
le leave-$(n-n_e)$-out, 
il n'est pas n\'ecessaire de prendre $V$ tr\`es grand pour avoir une performance 
tr\`es proche de l'optimum (\`a $n_e$ fix\'ee). 

\medbreak

\SAindex{moindres carres@moindres carr\'es!estimation de densit\'e|(}% 
Pour pouvoir \'enoncer des r\'esultats quantitatifs pr\'ecis 
sur la variance des incr\'ements, comme en section~\ref{VC.sec.estim-risque.var}, 
consid\'erons le cadre de l'estimation de densit\'e 
avec le co\^ut des moindres carr\'es et des estimateurs par 
histogrammes r\'eguliers de pas $h_m>0$ \citep{arlo_2016}. 
Alors, les r\'esultats \'enonc\'es en section~\ref{VC.sec.estim-risque.var} 
se g\'en\'eralisent, seuls les termes $\cW_i(h_m,P)$ devant \^etre chang\'es en 
$\cW_i(h_m, h_{m'},P)$. 
Toutes les comparaisons qualitatives entre m\'ethodes \'enonc\'ees 
en section~\ref{VC.sec.estim-risque.var}  
sont donc encore valables pour la s\'election d'estimateurs. 

La nouveaut\'e vient des aspects quantitatifs. 
Si $h_m$ et $h_{m'}$ sont de \og{}bonnes\fg{} valeurs 
(suffisamment diff\'erentes pour qu'il soit utile de savoir choisir entre les deux), 
alors $\cW_2( h_m, h_{m'}, P) \ll \cW_2 (h_m, P)$ : 
la variance des incr\'ements est donc beaucoup plus faible que la variance 
des crit\`eres, un ph\'enom\`ene d\'ej\`a remarqu\'e par \citet[sections~5.1 et 7.3]{brei_1992} 
dans un autre cadre. 

Par ailleurs, $n^{-1} \cW_2( h_m, h_{m'}, P)$ est alors du m\^eme ordre de grandeur que 
$n^{-2} \cW_1( h_m, h_{m'}, P)$. 
L'ordre de grandeur de la variance des incr\'ements est donc aussi impact\'e par les 
valeurs de $C_1^{\mathrm{vf}}(V,n)$ ou $C_1^{\mathrm{MC}} (V,n,n_e)$. 
Ceci change la donne pour la validation crois\'ee $V$-fold : 
\SAindex{validation croisee@validation crois\'ee!V-fold@$V$-fold}% 
\`a $V$ fix\'e, la variance des incr\'ements diminue au premier ordre ! 
Cependant, cette diminution reste de l'ordre d'une constante num\'erique 
(au plus $C_1^{\mathrm{vf}}(2,n)$ qui tend vers $10$ quand 
$n$ tend vers l'infini). 
De plus, vu la formule donnant $C_1^{\mathrm{vf}}(V,n)$, 
il suffit de prendre $V=5$ ou $10$ 
pour que la variance soit tr\`es proche de son minimum. 
Nous renvoyons aux r\'esultats de la fin de la section~\ref{VC.sec.estim-risque.var} 
pour discuter le cas Monte-Carlo 
et la comparaison avec le $V$-fold 
(\`a r\'einterpr\'eter en tenant compte de la diff\'erence d'ordre de grandeur des 
$\cW_i$ quand on consid\`ere les incr\'ements). 
Les exp\'eriences num\'eriques de \citet[section~6]{arlo_2016} permettent \'egalement 
de bien visualiser les ordres de grandeur des termes $\cW_i$ dans des cas r\'ealistes.

\subsection{Analyse au deuxi\`eme ordre : surp\'enalisation}
\label{VC.sec.sel-estim.surpen}

La deuxi\`eme partie de la section~\ref{VC.sec.sel-estim.diff-avec-estim-risq} 
met en \'evidence un deuxi\`eme facteur important \`a prendre en compte au deuxi\`eme ordre, 
en plus de la variance des incr\'ements : 
le fait que \og{}surp\'enaliser\fg{} 
un peu am\'eliore souvent les performances. 

\SAindex{validation croisee@validation crois\'ee!V-fold@$V$-fold|(}% 
\SAindex{validation croisee@validation crois\'ee!corrig\'ee|(}% 
Prenons le cas de l'estimation de densit\'e par moindres carr\'es. 
Alors, en comparant la validation crois\'ee $V$-fold et la validation crois\'ee 
$V$-fold corrig\'ee, 
\citet{arlo_2016} montrent que la validation crois\'ee $V$-fold surp\'enalise 
d'un facteur : 
\[ 1 + \frac{1}{2 (V-1)} \, . \]
Plus g\'en\'eralement, si l'on suppose 
\eqref{VC.eq.hyp-risque-n-OLS} et \eqref{VC.eq.hyp-penid-n-OLS} 
v\'erifi\'ees avec $\gamma(m) = 2 \beta(m)$ 
(comme c'est le cas en estimation de densit\'e par moindres carr\'es), 
alors, la validation crois\'ee avec un \'echantillon d'entra\^inement de taille $n_e$ 
surp\'enalise d'un facteur :  
\[ 
\frac{1}{2} \parenj{ 1 + \frac{n}{n_e} } > 1 
\, . 
\]

Lorsque surp\'enaliser est b\'en\'efique, \`a variance constante, 
il est donc pr\'ef\'erable de prendre $n_e \approx \kappa n$ 
avec $\kappa \in \,\, ]0,1[$ fonction du facteur de surp\'enalisation optimal. 
Par exemple, si l'optimum est de surp\'enaliser d'un facteur $1,12$ 
(comme dans le cadre de la figure~\ref{VC.fig.surpen}), 
on obtient que le mieux est d'avoir $n_e = n/1,24$,  
ce qui correspond au $5$-fold. 
Si l'optimum est de surp\'enaliser d'un facteur $3/2$, 
le mieux est d'avoir $n_e = n/2$,  
ce qui correspond au $2$-fold. 
Et lorsque le facteur de surp\'enalisation optimal est plus grand, 
il faut prendre $n_e$ encore plus petit,  
ce qui est impossible avec une m\'ethode \og{}$V$-fold\fg{} ! 

Attention toutefois \`a ne pas oublier que ceci 
n'est correct que si l'on compare des crit\`eres de \emph{m\^eme variance}. 
Si l'on augmente la variance pour surp\'enaliser 
(comme on le fait en diminuant $V$ pour le $V$-fold), 
on gagne d'un c\^ot\'e mais on perd de l'autre. 
\textit{In fine}, de nombreux cas de figure peuvent se produire quand on trace 
le risque du pr\'edicteur final obtenu par validation crois\'ee 
$V$-fold en fonction de $V$ : 
il peut d\'ecro\^itre avec $V$, augmenter avec $V$, 
ou \^etre minimal pour une valeur de $V$ \og{}interm\'ediaire\fg{}, 
comme le montrent des simulations num\'eriques \citep{arlo_2016}. 

Si l'on veut y voir clair, il faut raisonner en changeant un seul facteur \`a la fois 
parmi le nombre $V$ de d\'ecoupages (qui influe sur la variance) 
et la taille $n_e$ de l'\'echantillon d'entra\^inement (qui d\'efinit la surp\'enalisation, 
en influant assez peu sur la variance). 
\SAindex{validation croisee@validation crois\'ee!Monte-Carlo|(}% 
Ceci peut se faire naturellement avec la validation crois\'ee Monte-Carlo. 
Peut-on le faire avec une strat\'egie \og{}$V$-fold\fg{}, 
qui est un peu meilleure que la strat\'egie Monte-Carlo 
en termes de variance ?  
Oui, avec la p\'enalisation $V$-fold 
\SAindex{selection modeles@s\'election de mod\`eles!p\'enalisation|(}% 
\citep{arlo_2008,arlo_2016}. 
Sans la d\'efinir pr\'ecis\'ement ici, d\'ecrivons-en le principe : 
il s'agit d'une m\'ethode de p\'enalisation 
\[ 
\mh_C \bigl( D_n; (B_j)_{1 \leq j \leq V} \bigr) \in 
\argmin_{m \in \cM_n} \Bigl\{ \cRh_n \bigl( \fh_m (D_n) \bigr) 
+ C \penal_{\mathrm{vf}} \bigl( m; D_n; (B_j)_{1 \leq j \leq V} \bigr) \Bigr\}
\]
d\'efinie pour une partition $(B_j)_{1 \leq j \leq V}$ 
en $V$ blocs de m\^eme taille, 
qui co\"incide avec la validation crois\'ee $V$-fold corrig\'ee si $C=1$. 
En estimation de densit\'e par moindres carr\'es, elle co\"incide avec la validation 
crois\'ee $V$-fold si 
\[ 
C = 1 + \frac{1}{2(V-1)}
\, . 
\]
Alors, on peut choisir pour $C$ un \og{}bon\fg{} facteur de surp\'enalisation 
(en supposant que l'on conna\^it une telle valeur), 
et ensuite on choisit $V$ aussi grand que possible afin de minimiser la variance. 

\begin{remarque}[P\'enalisation $V$-fold ou validation crois\'ee Monte-Carlo ?]
\label{VC.penVF-vs-VCMC}
%
%%\hfill \\ 
On a propos\'e deux strat\'egies pour d\'ecoupler $V$ et $n_e$ : 
la validation crois\'ee Monte-Carlo et la p\'enalisation $V$-fold. 
Laquelle choisir en pratique ? 
Au vu des calculs de variance de \citet{arlo_2016}, 
la p\'enalisation $V$-fold semble pr\'ef\'erable : 
pour une m\^eme valeur de $V$ et de la constante de surp\'enalisation, 
on obtient le plus souvent une variance plus faible. 
Cependant, ceci est valable dans un cadre bien pr\'ecis, 
et il n'est pas certain que cela soit toujours vrai hors de ce cadre. 
Pire encore, le fait que la p\'enalisation $V$-fold 
surp\'enalise bien d'un facteur $C$ 
(en particulier, le fait qu'elle estime sans biais le risque lorsque $C=1$) 
n'est vraisemblablement pas valable lorsque l'hypoth\`ese \eqref{VC.eq.hyp-penid-n-OLS} 
n'est pas v\'erifi\'ee. 
Au final, il semble raisonnable d'utiliser la p\'enalisation $V$-fold 
pour la s\'election parmi des pr\'edicteurs \og{}r\'eguliers\fg{} 
--- c'est-\`a-dire, susceptibles de v\'erifier \eqref{VC.eq.hyp-penid-n-OLS}, 
au moins approximativement. 
Dans les autres cas, mieux vaut utiliser la validation crois\'ee Monte-Carlo 
(ou $V$-fold r\'ep\'et\'ee, si c'est possible). 
\SAindex{validation croisee@validation crois\'ee!V-fold repetee@$V$-fold r\'ep\'et\'ee}% 
Et quoi qu'il en soit, il ne faut pas oublier de bien choisir $V$ et $C$ ou $n_e$, 
selon le cas. 
\end{remarque}
\SAindex{selection modeles@s\'election de mod\`eles!p\'enalisation|)}% 
\SAindex{validation croisee@validation crois\'ee!Monte-Carlo|)}% 
\SAindex{validation croisee@validation crois\'ee!V-fold@$V$-fold|)}% 
\SAindex{validation croisee@validation crois\'ee!corrig\'ee|)}% 
\SAindex{moindres carres@moindres carr\'es!estimation de densit\'e|)}% 

\section{Conclusion}
\label{VC.sec.concl}
\SAindex{risque!estimation@estimation d'un risque|(}% 

\subsection{Choix d'une proc\'edure de validation crois\'ee}
\label{VC.sec.concl.choix}

Comment choisir une proc\'edure de validation crois\'ee pour un probl\`eme donn\'e ? 

Tout d'abord, lorsque le temps de calcul disponible est limit\'e, 
il faut prendre en compte la complexit\'e algorithmique 
des proc\'edures de validation crois\'ee. 
Dans la plupart des cas, 
celle-ci est proportionnelle au nombre $V$ de d\'ecoupages 
(il faut mettre en \oe uvre $V$ fois chaque r\`egle $\fh_m$ consid\'er\'ee). 
Parfois, les crit\`eres de validation crois\'ee peuvent \^etre 
calcul\'es plus efficacement, comme d\'etaill\'e par \citet[section~9]{arlo_2010}. 

\medbreak

\SAindex{moindres carres@moindres carr\'es!estimation de densit\'e}% 
Dans les cas \og{}r\'eguliers\fg{}, comme l'estimation de densit\'e par moindres carr\'es, 
les r\'esultats d\'ecrits en section~\ref{VC.sec.sel-estim} 
sugg\`erent la strat\'egie suivante. 
Dans un premier temps, choisir un facteur de surp\'enalisation 
($1$~si l'on veut estimer le risque ; 
souvent un peu plus pour un probl\`eme de s\'election d'estimateurs). 
\SAindex{validation croisee@validation crois\'ee!Monte-Carlo}% 
\SAindex{validation croisee@validation crois\'ee!V-fold repetee@$V$-fold r\'ep\'et\'ee}% 
Si l'on utilise la validation crois\'ee Monte-Carlo 
(ou le $V$-fold r\'ep\'et\'e), ceci se fait en choisissant la taille $n_e$ 
de l'\'echantillon d'entra\^inement. 
\SAindex{validation croisee@validation crois\'ee!V-fold@$V$-fold}% 
Si l'on utilise la p\'enalisation $V$-fold (d\'efinie en section~\ref{VC.sec.sel-estim.surpen}, 
voir notamment la remarque~\ref{VC.penVF-vs-VCMC}), 
ceci se fait directement avec le param\`etre~$C$. 
Puis on choisit le nombre $V$ de d\'ecoupages, 
en le prenant aussi grand que possible pour optimiser 
la performance statistique, dans la limite des capacit\'es de calcul. 
Les calculs de variance report\'es en section~\ref{VC.sec.sel-estim.var} 
indiquent qu'il n'est pas n\'ecessaire de prendre $V$ tr\`es grand 
pour avoir une performance quasi optimale. 
Avec une pr\'ecision importante : ceci n'est totalement vrai que si $n_e$ 
est de l'ordre de $\kappa n$ avec $\kappa \in \,\, ]0,1[$. 
Sinon (par exemple lorsque $n_e = n-1$), 
la variance de la validation simple peut \^etre \emph{beaucoup} plus grande que celle 
du leave-$(n-n_e)$-out, 
et alors il est n\'ecessaire de prendre $V$ tr\`es grand pour obtenir une variance 
du bon ordre de grandeur. 

\SAindex{validation croisee@validation crois\'ee!V-fold@$V$-fold}% 
Si l'on s'impose d'utiliser la validation crois\'ee $V$-fold, 
choisir $V$ peut s'av\'erer plus d\'elicat, 
car ce param\`etre d\'etermine simultan\'ement le facteur de surp\'enalisation 
(via la taille $n(V-1)/V$ de l'\'echantillon d'entra\^inement) 
et le nombre de d\'ecoupages. 
Pour la s\'election d'estimateurs, en admettant que dans la plupart des cas 
il est bon de surp\'enaliser \og{}un peu\fg{} (comme on le constate empiriquement), 
alors prendre $V$ entre $5$ et $10$ est un tr\`es bon choix (voire optimal). 
En effet, on surp\'enalise alors d'un facteur entre $1,05$ et $1,12$ 
%%% $1+1/18$ \`a $1+1/8$
et l'on a une variance proche de sa valeur minimale possible. 
Cette fourchette de valeurs de $V$ correspond d'ailleurs 
aux conseils classiques dans la litt\'erature statistique\footnote{% 
C'est par exemple le conseil donn\'e par \citet[section~7.10.1]{hast-2009}. 
Les intuitions g\'en\'eralement propos\'ees pour \'etayer ce conseil sont cependant 
diff\'erentes des arguments donn\'es dans ce texte, 
et parfois en contradiction avec les r\'esultats th\'eoriques 
rapport\'es dans ce texte. 
Par exemple, il est faux de dire que le $5$-fold a \emph{toujours} une variance plus faible 
que le leave-one-out.}. 
Si l'objectif est d'estimer le risque d'un pr\'edicteur, 
la situation est un peu diff\'erente car il faut choisir $V$ 
le plus grand possible pour minimiser le biais (et la variance) : 
les formules donnant biais et variance en fonction de $V$ 
permettent alors d'\'evaluer o\`u se situe le meilleur compromis entre 
pr\'ecision statistique et complexit\'e algorithmique. 

\medbreak

\SAindex{stabilit\'e}% 
Dans le cas g\'en\'eral, en particulier quand on s'int\'eresse \`a un ou plusieurs 
pr\'edicteurs \og{}instables\fg{}, 
le comportement des proc\'edures de validation crois\'ee peut \^etre diff\'erent, 
ce dont t\'emoignent plusieurs r\'esultats empiriques rapport\'es par 
\citet[sections 5.2 et~8]{arlo_2010}, en particulier 
ceux de \citet[section~7]{brei_1996}. 
Il semble alors pr\'ef\'erable de r\'ealiser des exp\'eriences num\'eriques 
(par exemple, sur des donn\'ees synth\'etiques) 
pour d\'eterminer le comportement de la validation crois\'ee 
en fonction de $n_e$ et~$V$. 
Au vu des sections \ref{VC.sec.estim-risque} et~\ref{VC.sec.sel-estim}, 
trois quantit\'es cl\'e sont \`a \'etudier. 
D'une part, le risque moyen $\E[\cR_P(\fh_m(D_n))]$ 
et sa d\'ependance en la taille $n$ de l'\'echantillon 
permettent de comprendre le biais de la validation crois\'ee 
(et son comportement au premier ordre pour la s\'election d'estimateurs). 
D'autre part, pour prendre en compte la variance, 
la quantit\'e \`a calculer d\'epend de l'objectif. 
Si l'on s'int\'eresse \`a l'estimation du risque, 
alors il faut calculer la variance de l'estimateur par validation crois\'ee 
du risque d'une r\`egle~$\fh_m$. 
Si l'on s'int\'eresse \`a la s\'election d'estimateurs, 
alors la quantit\'e \`a calculer est 
\begin{equation*}
\begin{split}
\var \Bigl( 
 \cRh^{\mathrm{vc}} \bigl( \fh_m; D_n; (E_j)_{1 \leq j \leq V} \bigr) 
 - \cRh^{\mathrm{vc}} \bigl( \fh_{m^{\circ}}; D_n; (E_j)_{1 \leq j \leq V} \bigr)
\Bigr)
\\ \text{o\`u} \qquad 
m^{\circ} \in \argmin_{m \in \cM_n} \biggl\{ \E \Bigl[ \cR_P \bigl( \fh_m(D_n) \bigr) \Bigr] \biggr\}
\end{split}
\end{equation*}
est le choix oracle et $m$ est \og{}proche\fg{} de l'oracle. 
Notons que l'on peut s'\'epargner certains calculs de variance en utilisant 
la proposition~\ref{VC.pro.var-MCCV} et 
la formule \eqref{VC.eq.var-hold-out} en section~\ref{VC.sec.estim-risque.var}, 
ainsi que leurs g\'en\'eralisations aux incr\'ements d'estimateurs par validation crois\'ee.

\begin{remarque}[Objectif d'identification]
\label{VC.rk.identification}
\SAindex{selection modeles@s\'election de mod\`eles!consistance en s\'election}% 
On aboutit \`a des conclusions diff\'erentes quand on a pour objectif 
d'\emph{identifier} 
le \og{}vrai\fg{} mod\`ele\footnote{% 
Le vrai mod\`ele est d\'efini comme le plus petit mod\`ele 
contenant $\fst$, en supposant qu'il existe.} % fin footnote
ou la r\`egle d'apprentissage dont l'exc\`es de risque d\'ecro\^it le plus 
vite parmi $(\fh_m)_{m \in \cM_n}$ 
\citep[section~7]{arlo_2010}. 
\SAindex{moindres carres@moindres carr\'es!regression lineaire@r\'egression lin\'eaire}% 
Par exemple, en r\'egression lin\'eaire par moindres carr\'es, 
la validation crois\'ee choisit le vrai mod\`ele avec probabilit\'e~$1$ 
asymptotiquement si et seulement si $n_e \ll n$ 
\citep[avec des hypoth\`eses sur $V$,][]{shao_1997}. 
Ce ph\'enom\`ene se g\'en\'eralise \`a d'autres cadres 
et a \'et\'e nomm\'e \og{}paradoxe de la validation crois\'ee\fg{} par 
\citet{yang_2006,yang_2007} : 
pour l'identification, plus on a d'observations \`a disposition, 
plus il faut en utiliser une fraction petite pour l'entra\^inement. 
On peut rapprocher ces r\'esultats du fait qu'il est n\'ecessaire 
de surp\'enaliser fortement pour une identification optimale, 
d'o\`u le fait que BIC fonctionne, alors que AIC et $C_p$ sont sous-optimales. 
\SAindex{BIC}% 
\SAindex{AIC|(}% 
\SAindex{cp@$C_p$|(}% 
Signalons enfin que pour l'identification, Yang conseille d'utiliser la variante 
\SAindex{validation croisee@validation crois\'ee!vote@(par) vote}% 
\og{}par vote majoritaire\fg{} 
de la validation crois\'ee, d\'efinie en section~\ref{VC.sec.def.var}. 
\end{remarque}

\subsection{Validation crois\'ee ou proc\'edure sp\'ecifique ?}

Lorsqu'une autre proc\'edure de s\'election d'estimateurs 
est disponible --- par exemple, $C_p$ ou AIC ---, 
faut-il la pr\'ef\'erer \`a la validation crois\'ee ? 

Si l'on est dans le cadre sp\'ecifique pour lequel cette proc\'edure 
a \'et\'e construite (pour $C_p$, la r\'egression lin\'eaire homosc\'edastique avec 
le co\^ut quadratique et des estimateurs des moindres carr\'es), 
\SAindex{moindres carres@moindres carr\'es!regression lineaire@r\'egression lin\'eaire}% 
\SAindex{regression@r\'egression!homoscedastique@homosc\'edastique}% 
alors c'est elle qu'il faut utiliser. 
Des exp\'eriences num\'eriques montrent en effet 
que la validation crois\'ee fonctionne alors souvent un peu 
moins bien que les proc\'edures sp\'ecifiques. 
C'est le prix de l'\og{}universalit\'e\fg{} de la validation crois\'ee. 

En revanche, si l'on risque de sortir un peu du cadre o\`u la proc\'edure 
\og{}sp\'ecifique\fg{} est connue pour \^etre optimale 
(par exemple, pour $C_p$, si l'on soup\c{c}onne les donn\'ees d'\^etre h\'et\'erosc\'edastiques), 
alors il est plus s\^ur d'utiliser une proc\'edure \og{}universelle\fg{} 
comme la validation crois\'ee 
(si les capacit\'es de calcul disponibles le permettent). 
\SAindex{AIC|)}% 
\SAindex{cp@$C_p$|)}% 

\subsection{Limites de l'universalit\'e}
Il est bon de garder en t\^ete que la validation crois\'ee ne peut pas fonctionner 
parfaitement d'une mani\`ere totalement universelle, 
ne serait-ce qu'\`a cause des r\'esultats \og{}on n'a rien sans rien\fg{} 
\citep[section~6]{Arl_2016_JESchap2}. 
En particulier, la validation crois\'ee suppose implicitement 
qu'il est possible d'\'evaluer le risque d'un pr\'edicteur 
$f \in \cF$ \`a partir de $n-n_e$ observations. 
Ce n'est clairement pas possible dans l'exemple construit pour d\'emontrer le 
premier r\'esultat \og{}on n'a rien sans rien\fg{} de 
\citet[th\'eor\`eme~3 en section~6]{Arl_2016_JESchap2}.

\medbreak

Des probl\`emes se posent \'egalement lorsque les hypoth\`eses explicitement faites 
par la validation crois\'ee sont viol\'ees  
(donn\'ees ind\'ependantes et de m\^eme loi ; 
et pour la s\'election d'estimateurs, la collection $(\fh_m)_{m \in \cM_n}$ est suppos\'ee 
n'\^etre \og{}pas trop grande\fg{}) \citep[section~3.9]{Arl_2016_JESchap2}. 
\begin{itemize}
\item 
Lorsque les donn\'ees sont \emph{d\'ependantes}, 
les \'echantillons d'entra\^inement et de validation ne sont plus n\'ecessairement 
ind\'ependants, 
ce qui peut induire un biais assez fort pour l'estimation du risque 
d'une r\`egle d'apprentissage. 
Dans le cas de d\'ependance \`a courte port\'ee, ce biais peut \^etre \'evit\'e 
en utilisant des \'echantillons d'entra\^inement $D_n^{E}$ et 
de validation $D_n^{V}$ tels que $E$ et $V$ sont suffisamment \'eloign\'es. 
\item 
Pour une s\'erie temporelle \emph{non-stationnaire}, 
si l'on s'int\'eresse \`a la pr\'evision du \og{}futur\fg{} \`a partir du \og{}pass\'e\fg{}, 
on ne peut pas utiliser la validation crois\'ee telle quelle. 
En choisissant un \'echantillon d'entra\^inement dans le pass\'e 
par rapport \`a l'\'echantillon de validation 
(et en les \'eloignant si besoin), 
on peut appliquer la validation simple, 
mais sans garantie en raison de la non-stationnarit\'e. 
Si l'on a observ\'e une s\'erie assez longue, on peut \'egalement utiliser une fen\^etre glissante 
pour multiplier les couples entra\^inement/validation. 
\item 
\SAindex{donnees aberrantes@donn\'ees aberrantes}% 
En pr\'esence de \emph{donn\'ees aberrantes}, 
il est indispensable d'utiliser des pr\'edicteurs robustes et/ou une fonction de co\^ut robuste. 
\item 
Pour la s\'election d'estimateurs parmi une \emph{collection exponentielle} 
\citep[section~3.9]{Arl_2016_JESchap2}, 
\SAindex{selection estimateurs@s\'election d'estimateurs!principe d'estimation sans biais du risque}% 
le principe d'estimation sans biais du risque ne fonctionne plus 
et bon nombre de conclusions de ce texte sont erron\'ees. 
Une id\'ee est alors de \og{}surp\'enaliser\fg{} fortement, 
en prenant un \'echantillon d'entra\^inement de petite taille 
(comme indiqu\'e en section~\ref{VC.sec.sel-estim}), 
ce qui permet d'avoir une proc\'edure de la deuxi\`eme famille d\'ecrite 
par \citet[section~3.9]{Arl_2016_JESchap2}. 
Une deuxi\`eme strat\'egie --- appel\'ee \og{}validation crois\'ee en deux \'etapes\fg{} ou  
\og{}double cross\fg{} \citep{stone-1974} --- est la suivante. 
D'abord, on forme un nombre polynomial de groupes de r\`egles d'apprentissage. 
Puis, \`a l'int\'erieur de chaque groupe, on utilise la validation crois\'ee pour s\'electionner une r\`egle ; 
on dispose donc d'une \og{}m\'eta-r\`egle\fg{} associ\'ee \`a chaque groupe. 
Enfin, on s\'electionne par validation crois\'ee l'une de ces m\'eta-r\`egles 
(en suivant les conseils formul\'es \`a la fin de la section~\ref{VC.sec.def.gal}). 
Cette deuxi\`eme strat\'egie s'applique par exemple pour le probl\`eme de d\'etection de ruptures. 
\SAindex{detection de ruptures@d\'etection de ruptures}% 
\end{itemize}
\citet[section~8]{arlo_2010} d\'etaillent tous ces points 
et donnent des r\'ef\'erences bibliographiques. 
\SAindex{risque!estimation@estimation d'un risque|)}% 
\SAindex{selection estimateurs@s\'election d'estimateurs|)}% 

\section{Annexe : exercices}

%%% Regles intelligentes

\begin{exercice}
\label{VC.exo.intelligente-majorite}
\SAindex{regle apprentissage@r\`egle d'apprentissage!intelligente|(}% 
\SAindex{partition!regle classification@r\`egle de classification|(}% 
On se place en classification $0$--$1$. 
Soit $\fh$ la r\`egle par partition associ\'ee \`a la partition 
triviale $\cA = \{ \cX \}$; 
autrement dit, $\fh( (x_i,y_i)_{1 \leq i \leq n} ; x)$ 
r\'ealise un vote majoritaire parmi les $y_i$, 
sans tenir compte des $x_i$ ni de $x$. 
\\
D\'emontrer que $\fh$ n'est pas intelligente. 
En d\'eduire qu'aucune r\`egle par partition (sur une partition $\cA$ fixe quand $n$ varie) 
n'est intelligente. 
\\
Modifier l\'eg\`erement $\fh$ pour la rendre intelligente. 
\end{exercice}

\begin{exercice}
\label{VC.exo.intelligente-part-random}
On se place en classification $0$--$1$. 
Soit $\widetilde{f}$ la r\`egle par partition associ\'ee \`a la partition 
triviale $\cA = \{ \cX \}$, 
avec une d\'ecision \og{}randomis\'ee\fg{} dans les cas d'\'egalit\'e : 
\SAindex{regle apprentissage@r\`egle d'apprentissage!randomis\'ee}% 
s'il y a exactement $n/2$ des $y_i$ qui sont \'egaux \`a~$1$, 
$\widetilde{f}( (x_i,y_i)_{1 \leq i \leq n} ; x)$ vaut $1$ avec probabilit\'e~$1/2$ 
et vaut $0$ avec probabilit\'e~$1/2$. 
La notion de r\`egle intelligente s'\'etend naturellement \`a $\widetilde{f}$ 
en prenant aussi l'esp\'erance sur la randomisation interne $\widetilde{f}$  
dans la d\'efinition de son risque moyen. 
\\
D\'emontrer que $\widetilde{f}$ est intelligente. 
En d\'eduire que tout r\`egle par partition (sur une partition $\cA$ fixe quand $n$ varie), 
randomis\'ee de la m\^eme mani\`ere en cas d'\'egalit\'e, est intelligente.  
\end{exercice}
\SAindex{regle apprentissage@r\`egle d'apprentissage!intelligente|)}%  
\SAindex{partition!regle classification@r\`egle de classification|)}% 

%%% Variance 

\begin{exercice}
\label{VC.exo.var-VF-repete}
\SAindex{validation croisee@validation crois\'ee!V-fold repetee@$V$-fold r\'ep\'et\'ee}% 
D\'emontrer la remarque~\ref{VC.rk.var.pro-ho-lpo-et-MCCV.generalisation} 
en section~\ref{VC.sec.estim-risque.var}. 
En particulier, si 
$(B_j^{\ell})_{1 \leq j \leq V}$, 
$\ell \in \{1, \ldots, L\}$, 
est une suite de partitions de $\{1, \ldots, n\}$,  
ind\'ependantes et de m\^eme loi uniforme sur l'ensemble 
des partitions de $\{1, \ldots, n\}$ en $V$ blocs de m\^eme taille, 
ind\'ependante de $D_n$, 
d\'emontrer la formule suivante pour la variance de l'estimateur 
par validation crois\'ee $V$-fold r\'ep\'et\'ee 
du risque d'une r\`egle d'apprentissage $\fh_m$ : 
\begin{align*}
&\qquad \var \biggl( 
\cRh^{\mathrm{vc}} \Bigl( \fh_m ; D_n ; \bigl( ( B_j^{\ell})^c \bigr)_{1 \leq j \leq V, 1 \leq \ell \leq L} \Bigr)  
\biggr) 
\\
&= 
\var\parenj{ \cRh^{\mathrm{lpo}}\parenj{ \fh_m; D_n; \frac{n (V-1)}{V} } } 
\\
&\qquad + \frac{1}{L} \biggl[ 
\underbrace{ \var \Bigl( \cRh^{\mathrm{vf}} \bigl( \fh_m ; D_n ; (B_j^{1})_{1 \leq j \leq V} \bigr)  \Bigr) 
- \var\parenj{ \cRh^{\mathrm{lpo}}\parenj{ \fh_m; D_n; \frac{n (V-1)}{V} } }  
}_{\geq 0}
\biggr]
\, .
\end{align*}
\end{exercice}

\begin{exercice}
\label{VC.exo.eq.var-hold-out}
D\'emontrer la formule \eqref{VC.eq.var-hold-out} 
donnant la variance du crit\`ere par validation simple. 
\SAindex{validation simple}% 
\end{exercice}

\section*{Remerciements}
Cet texte fait suite \`a un cours donn\'e dans le cadre des Journ\'ees 
d'\'Etudes en Statistique 2016. 
Il s'agit d'une version pr\'eliminaire du chapitre~3 
du livre \emph{Apprentissage statistique et donn\'ees massives}, 
\'edit\'e par Fr\'ed\'eric Bertrand, Myriam Maumy-Bertrand, 
Gilbert Saporta et Christine Thomas-Agnan, 
\`a para\^itre aux \'editions Technip. 

Je remercie vivement tous les coll\`egues avec qui j'ai travaill\'e sur ce sujet, 
en particulier mes coauteurs Alain Celisse, Matthieu Lerasle et Nelo Magalh\~aes. 
Je remercie \'egalement 
Matthieu Lerasle pour avoir relu ce texte, 
ainsi que 
les \'etudiants qui ont suivi mon cours de master 2 
\og{}apprentissage statistique et r\'e\'echantillonnage\fg{} 
\`a l'Universit\'e Paris-Sud 
et les participants des JES 2016 pour leurs questions et commentaires.

\SAindex{validation croisee@validation crois\'ee|)}%

\bibliographystyle{plainnat-fr}
\bibliography{hal_biblio}

\printindex

\end{document}